\documentclass[11pt
,twoside]{elsarticle}
\usepackage{epsfig,amsthm,
amssymb,
amsmath,
amsfonts,subeqnarray
}
\usepackage{setspace}
\usepackage{verbatim,comment}
\usepackage{colortbl}

\begin{document}


\definecolor{red}{rgb}{.0, 0., 0.}
\definecolor{blue}{rgb}{0.0, 0.0, .0}

\newcommand{\rb}[1]{\raisebox{1.5ex}[0pt]{#1}}
\newcommand{\be}{\begin{eqnarray}}
\newcommand{\ee}{\end{eqnarray}}
\newcommand{\bal}{\begin{aligned}}
\newcommand{\eal}{\end{aligned}}
\newcommand{\bes}{\begin{eqnarray*}}
\newcommand{\ees}{\end{eqnarray*}}
\newcommand{\bs}{\begin{subeqnarray}}
\newcommand{\es}{\end{subeqnarray}}
\newcommand{\bss}{\begin{subeqnarray*}}
\newcommand{\ess}{\end{subeqnarray*}}
\newcommand{\tb}[1]{|\hskip -.08em|\hskip-.08em|#1|\hskip-.08em|\hskip-.08em|}

\newtheorem{theorem}{Theorem}[section]
\newtheorem{lemma}[theorem]{Lemma}
\newtheorem{proposition}[theorem]{Proposition}
\newtheorem{remark}[theorem]{Remark}
\newtheorem{corollary}[theorem]{Corollary}
\newtheorem{definition}[theorem]{Definition}

\newcommand{\steref}[1]{Step~\ref{#1}}
\newcommand{\defref}[1]{Definition~\ref{#1}}
\newcommand{\thmref}[1]{Theorem~\ref{#1}}
\newcommand{\secref}[1]{Section~\ref{#1}}
\newcommand{\subsecref}[1]{\S~\ref{#1}}
\newcommand{\lemref}[1]{Lemma~\ref{#1}}
\newcommand{\propref}[1]{Proposition~\ref{#1}}
\newcommand{\rmkref}[1]{Remark~\ref{#1}}
\newcommand{\figref}[1]{Figure~\ref{#1}}
\newcommand{\figrefs}[2]{Figure2~\ref{#1} and ~\ref{#2}}
\newcommand{\tabref}[1]{Table~\ref{#1}}
\newcommand{\Fig}[1]{Figure~\ref{#1}}
\numberwithin{equation}{section}




\def\for{\quad{for}\ }
\def\with{\quad{with}\ }
\def\inn{\quad{in}\ }
\def\nm#1#2{\|#1\|_{#2}}

\def\on{\quad{on}\ }
\def\orr{\quad{or}\ }
\def\andy{\quad{and}\ }
\def\where{\quad{where}\ }
\def\when{\quad{when}\ }
\def\ifff{\quad{if}\ }

\def\hf{\widehat f}
\def\hu{\widehat u}
\def\tw{\widetilde w}
\def\z{\zeta}

\def\dt{\bar\partial}
\def\dsum{\displaystyle\sum}
\def\la{\lambda}
\def\ka{\kappa}
\def\Si{\Sigma}
\def\Om{\Omega}
\def\O{\Omega}
\def\om{\omega}
\def\ep{\epsilon}
\def\vep{\varepsilon}
\def\fy{\varphi}
\def\al{\alpha}

\def\de{\delta}
\def\De{\Delta}
\def\th{\theta}
\def\gk{{\Gamma_{\al,k}}}
\def\ga{\gamma}
\def\Ga{\Gamma}
\def\G{\Gamma}
\def\as{\quad\text{as}\ }
\def\Re{\operatorname{Re}}
\def\Im{\operatorname{Im}}
\def\span{\operatorname{span}}
\def\max{\operatorname{max}}
\def\trace{\operatorname{trace}}
\def\det{\operatorname{det}}
\def\min{\operatorname{min}}
\def\and{\quad\text{and}\quad}
\def\i{\text{\rm i}}
\def\({\left(}
\def\){\right)}
\def\bF{\mathbf F}
\def\bG{\mathbf G}
\def\bV{\mathbf V}
\def\bu{\mathbf u}
\def\bv{\mathbf v}
\def\bw{\mathbf w}
\def\by{\mathbf y}
\def\b0{\mathbf 0}

\renewcommand{\hat}{\widehat}

\begin{frontmatter}

\title{Turing Instability for a Ratio-Dependent Predator-Prey Model with
Diffusion}
\author[local,perm]{Shaban Aly\fnref{fn1}}
\ead{shhaly70@yahoo.com}

\author[local]{Imbunm Kim\fnref{fn2}}
\ead{ikim@snu.ac.kr}

\author[local,joint]{Dongwoo Sheen\corref{cor1}\fnref{fn3}}
\ead{sheen@snu.ac.kr}
\ead{Tel:08228806543/Fax:08228874694}

\address[local]{Department of Mathematics, Seoul National University, Seoul
 151-747,Korea.}
\address[perm]{ permanent address:Department of Mathematics, Faculty of Science, Al-Azhar University, Assiut 71511, Egypt.}
\address[joint]{Interdisciplinary Program in Computational Science \& Technology,\\
  Seoul National University, Seoul 151-747,Korea}


\begin{abstract}
Ratio-dependent predator-prey models have been increasingly favored by field
ecologists where predator-prey interactions have to be taken into account
the process of predation search.
In this paper we study the conditions of the existence and stability properties of the
equilibrium solutions in a reaction-diffusion model in which predator mortality
is neither a constant  nor an unbounded function,
but it is increasing with the predator abundance. We show that analytically at a
certain critical value a diffusion driven (Turing type) instability occurs,
{\it i.e.} the stationary solution stays stable with respect to the kinetic system
(the system without diffusion). We also show that the stationary solution
becomes unstable with respect to the  system with diffusion and that
Turing bifurcation takes place: a spatially
non-homogenous (non-constant) solution (structure or pattern) arises.
A numerical scheme that preserve the positivity of the numerical solutions
and the boundedness of prey solution will be presented.
Numerical examples are also included.

\end{abstract}


\end{frontmatter}

\section{Introduction}
Since it is rare to find a pair of biological species in nature which meet precise prey-dependence or ratio-dependence
functional responses in predator-prey models,
especially when predators have to search for food (and therefore,
have to share or compete for food), a more suitable general predator-prey
theory should be based on the so-called ratio-dependent theory (see
\cite{Berezovskaya, Hsu, Jost, Kuang}). The theory may be stated as follows:
the per capita predator growth rate should be a function of
the ratio of prey to predator abundance, and so should be the so-called
predator functional response. Such cases are strongly supported by numerous field
and laboratory experiments and observations (see, for instance,
\cite{Arditi1, Arditi2, Arditi3, Arditi4}).

Denote by $N(t)$ and $P(t)$ the population densities of prey and predator at time $t,$ respectively.
Then the ratio-dependent type predator-prey model with Michaelis-Menten type
functional response is given as follows:
\begin{subeqnarray} \label{1.1}
\frac{dN}{dt} &=&rN\left(1-\frac{N}{K}\right)-\frac{aNP}{mP+N}, \\
\frac{dP}{dt} &=&P\left[ -Q(P)+\frac{bN}{mP+N}\right],
\end{subeqnarray}
where $a, b, m, K,$ and $r$ are positive constants.
In \eqref{1.1}, $Q(P)$ denotes a mortality function of predator, and
$r$ and $K$ the prey growth rate with intrinsic growth rate and carrying
capacity in the absence of predation, respectively, while $a, b,$ and $m$ are
model-dependent constants.

From a formal point of view, this model looks very similar to the well-known
Michaelis-Menten-Holling predator-prey model:
\begin{subeqnarray}\label{1.2}
\frac{dN}{dt} &=&rN\left(1-\frac{N}{K}\right)-\frac{aNP}{c+N},  \\
\frac{dP}{dt} &=&P\left[-Q(P)+\frac{bN}{c+N}\right].
\end{subeqnarray}
Indeed, the only difference between Models (\ref{1.1}) and (\ref{1.2}) is
that the parameter $c$ in (\ref{1.2}) is replaced by $mP$ in (\ref{1.1}).
Both terms $mP$ and $c$ are proportional to the so-called searching time of
the predator, namely, the time spent by each predator to find one prey.
Thus, in the Michaelis-Menten-Holling model \eqref{1.2} the searching time is assumed to
be independent of predator density, while in the ratio-dependent
Michaelis-Menten type model \eqref{1.1} it is proportional to predator
density ({\it i.e.}, other predators strongly interfere).

Predators and preys are usually abundant in space with different densities
at difference positions and they are diffusive.
Several papers have focused on the effect of diffusion which plays a crucial
role in permanence and stability of population (see \cite{baurmann, Casten,
lizana08, Grindrod, mukho,  Okubo, Pang}, and the references therein). Especially in \cite{mukho} the effect of variable
dispersion rates on Turing instability was extensively studied, and in
\cite{lizana08}
the dynamics of ratio-dependent system has been analyzed in details  with diffusion and delay terms included.
Cavani and Farkas (see \cite{Cavani}) have considered a modification of 
\eqref{1.2} when a diffusion was introduced, yielding:
\begin{subeqnarray}\label{1.3}
\frac{\partial N}{\partial t} &=&rN\left(1-\frac{N}{K}\right)-\frac{aNP}{c+N}+D_{1}%
\frac{\partial^{2}N}{\partial x^{2}},\quad x\in (0,l),t>0,
\\
\frac{\partial P}{\partial t} &=&P\left[
-Q(P)+\frac{bN}{c+N}\right]+D_{2}\frac{\partial
^{2}P}{\partial x^{2}},\quad x\in (0,l),t>0,
\end{subeqnarray}
where the specific mortality of the predator is given by
\begin{equation}
Q(P)=\frac{\gamma +\delta P}{1+P},  \label{1.4}
\end{equation}
which depends on the quantity of predator. Here, the positive constants $\gamma$ and $\delta$
denote the minimal mortality and the limiting mortality of the predator,
respectively. Throughout the paper, the following natural condition
\begin{equation}\label{assume-gamma-delta}
0<\gamma \leq \delta
\end{equation}
will be assumed, and we will consider the case of
the constant diffusivity, $D_{i}>0$, $i=1,2$. The advantage of this model
is that the predator mortality is neither a constant nor an
unbounded function, but still it is increasing with the predator abundance.
On the other hand,  combining \eqref{1.1} and \eqref{1.3}, 
many authors (see \cite{Bartumeus, Pang, Wanga}, 
for instance) have studied a more general model as follows:
\begin{subeqnarray*} 
\frac{\partial N}{\partial t} &=&rN\(1-\frac{N}{K}\)-\frac{aNP}{mP+N}+D_{1}%
\frac{\partial^{2}N}{\partial x^{2}},\quad x\in (0,l) ,t>0, 
\\
\frac{\partial P}{\partial t} &=&P\left[-Q(P)+\frac{bN}{mP+N}\right]+D_{2}\frac{%
\partial^{2}P}{\partial x^{2}},\quad x\in (0,l) ,t>0,
\end{subeqnarray*}%
with the specific mortality of the predator somewhat restricted in the form
\begin{equation}
Q(P)=d.  \label{1.6}
\end{equation}
In this paper we consider a ratio-dependent reaction-diffusion predator-prey
model with
Michaelis-Menten type functional response and
the specific mortality of the predator given by (\ref{1.4}) instead of \eqref{1.6}.
We study the effect of the diffusion on the stability of the stationary
solutions.
Also we explore under which parameter values Turing instability can occur
giving rise to non-uniform stationary solutions satisfying
the following equations:
\begin{subeqnarray}\label{full1}
\frac{\partial N}{\partial t} &=&rN\left(1-\frac{N}{K}\right)-\frac{aNP}{mP+N}+D_{1}%
\frac{\partial^{2}N}{\partial x^{2}},\quad x\in (0,l),t>0,
\\
\frac{\partial P}{\partial t} &=&P\left[ -\frac{\gamma +\delta P}{1+P}+\frac{%
bN}{mP+N}\right] +D_{2}\frac{\partial^{2}P}{\partial x^{2}},\quad x\in
(0,l),t>0,
\end{subeqnarray}%
assuming that prey and predator are diffusing according to Fick's
law in the interval $x\in \lbrack 0,l].$ We are interested in the
solutions $N, P:(l,0)\times \mathbb R^{+} \rightarrow \mathbb R^{+}$ fulfilling the Neumann boundary conditions
\begin{equation}
N_{x}(0,t)=N_{x}(l,t)=P_{x}(0,t)=P_{x}(l,t)=0,  \label{Cond-full1}
\end{equation}
and initial conditions
\begin{equation*}
N(x,0)\geq 0,\quad P(x,0)\geq 0,\quad x\in (0,l).
\end{equation*}%
For simplicity, we nondimensionalize the system (\ref{full1}) with
the following scaling
\begin{equation*}
\widetilde{t}=rt,\quad \widetilde{N}=\frac{N}{K},\quad \widetilde{P}=%
\frac{mP}{K},
\end{equation*}
and letting
\begin{equation*}
\alpha =\dfrac{a}{mr},\;\widetilde{\gamma }=\dfrac{\gamma }{b},\;\widetilde{%
\delta }=\dfrac{\delta }{b},\;\epsilon =\dfrac{b}{r},\;\beta =\frac{K}{m}%
,\;d_{1}=\dfrac{D_{1}}{r},\;d_{2}=\dfrac{D_{2}}{r}.
\end{equation*}%
For the sake of simplification of notations, dropping tildes,
the system \eqref{full1} takes the form
\begin{subeqnarray} \label{2.1}
\frac{\partial N}{\partial t} &=&N(1-N)-\frac{\alpha NP}{P+N}+d_{1}\frac{%
\partial^{2}N}{\partial x^{2}},\quad x\in (0,l),t>0, \\
\frac{\partial P}{\partial t} &=&\epsilon P\left[-\frac{\gamma +\delta \beta P}{%
1+\beta P}+\frac{N}{P+N}\right]+d_{2}\frac{\partial^{2}P}{\partial x^{2}},\quad
x\in (0,l),t>0.
\end{subeqnarray}%
Set
\[
\bF=\left(\begin{array}{c}F_1 \\F_2 \end{array} \right),
\bu:=\left(\begin{array}{c}N \\P \end{array} \right)
D:=\left(
\begin{array}{cc}
d_{1} & 0 \\
0 & d_{2}
\end{array}
\right),
\]
where
\begin{equation*}
F_{1}(N,P)=N(1-N)-\frac{\alpha NP}{P+N},\quad F_{2}(N,P)=\epsilon P\left[-\frac{%
\gamma +\delta \beta P}{1+\beta P}+\frac{N}{P+N}\right].
\end{equation*}%
Then the system \eqref{2.1} with the boundary conditions \eqref{Cond-full1}
takes the form
\begin{equation}
\bu_{t}=\bF(\bu)+D\frac{\partial^{2}\bu}{\partial x^{2}};\quad
\bu_{x}(0,t)=\bu_{x}(l,t)=\b0. \label{2.3} 
\end{equation}
Clearly, in case the predator and prey are spatially homogeneous,
the spatially constant solution $\bu(t)=(N(t),P(t))^T$
of (\ref{2.3}),
fulfilling the boundary conditions obviously,
satisfies the kinetic system
\begin{equation}
\bu_{t}=\bF(\bu).  \label{2.5}
\end{equation}

\section{The model without diffusion}
In this section we will study the system (\ref{2.1}) without diffusion,
{\it i.e.},
\begin{subeqnarray} \label{kinetic}
\frac{dN}{dt} &=&N(1-N)-\frac{\alpha NP}{P+N}, \\
\frac{dP}{dt} &=&\epsilon P\left[-\frac{\gamma +\delta \beta P}{1+\beta P}+\frac{%
N}{P+N}\right].
\end{subeqnarray}%
In particular, we will focus on the existence of equilibria
and their local stability. This information will be crucial in the next
section where we study the effect of the diffusion parameters on the
stability of the steady states.

The equilibria of the system (\ref{kinetic}) are given by the solution of
the following equations
\begin{equation*}
N(1-N)-\frac{\alpha NP}{P+N}=0,\quad \epsilon P\left(-\frac{\gamma +\delta \beta P%
}{1+\beta P}+\frac{N}{P+N}\right)=0.
\end{equation*}%
The system has at least one equilibrium with positive values. This is
the point of intersection of the prey null-cline
\begin{equation*}
P=H_{1}(N)=\frac{(1-N)N}{\alpha -(1-N)}
\end{equation*}%
and the predator null-cline
\begin{equation*}
P=H_{2}(N)=\frac
{\gamma-\beta(1-\delta )N+
2\sqrt{\left\{\gamma-\beta (1-\delta )N\right\}^{2}
-4\beta\delta(1-\gamma)N}}{2\beta\delta}.
\end{equation*}%
Thus, denoting the coordinates of a positive equilibrium by $(\overline{N},%
\overline{P})$, these coordinates satisfy $\overline{P}=H_{1}(\overline{N}%
)=H_{2}(\overline{N}).$

The Jacobian matrix of the system (\ref{kinetic}) linearized at $(\overline{N%
},\overline{P})$ is
\begin{equation}
A=\left(
\begin{array}{lr}
\Theta_{1} & -\Theta_{2} \\
\Theta_{3} & -\Theta_{4}%
\end{array}
\right),  \label{2.6}
\end{equation}
where
\begin{equation*}
\trace A=\Theta_{1}-\Theta_{4},\det A=\Theta_{2}\Theta_{3}-\Theta
_{1}\Theta_{4}
\end{equation*}
and
\begin{subeqnarray*}
\Theta_{1} &=&-\overline{N}+\frac{\alpha \overline{N}\,\overline{P}}{(%
\overline{P}+\overline{N})^{2}},
\quad \Theta_{2}=\frac{\alpha \overline{N}^{2}}{(\overline{P}+\overline{N})^{2}},
\\
\Theta_{3}&=&\frac{\epsilon
  \overline{P}^{2}}{(\overline{P}+\overline{N})^{2}},
\quad
\Theta_{4} =\frac{\epsilon \beta \overline{P}(\delta -\gamma )}{(1+\beta
\overline{P})^{2}}+\frac{\epsilon \overline{N}\,\overline{P}}{(%
\overline{P}+\overline{N})^{2}}.
\end{subeqnarray*}
The characteristic equation is given by
\begin{equation*}
\lambda^{2}-\(\trace A\)\lambda +\det A=0.
\end{equation*}
Recall that $(\overline{N},\overline{P})$ is
locally asymptotically stable if $\Re\lambda <0$, which is equivalent to have
$\trace A<0$ and $\det A>0$. For this, we will assume that
\begin{eqnarray}\label{assume-theta}
\Theta_{1}<\Theta_{4},\quad \Theta_{2}\Theta_{3} > \Theta_{1}\Theta_{4}.
\end{eqnarray}
\begin{remark}
Due to \eqref{assume-gamma-delta}, we see that $\Theta_4 > 0$. If $\Theta_1
\le 0$, then the two conditions in \eqref{assume-theta} hold.
\end{remark}

\section{The model with diffusion}
In this section we will investigate in
Turing instability and bifurcation for our model problem. We will also study
pattern formation of the predator-prey solutions.
\subsection{Local existence of solutions}
Before studying the stability of equilibrium solutions,
we will discuss about the local existence and uniqueness of solution
for a given ratio-dependent reaction-diffusion predator-prey model.
Applying the criteria for the local existence of solution (see
\cite{Morgan2, Morgan}) to the nonlinear parabolic systems (\ref{2.3}),
we see that there exists a unique local solution of the given system.

Let $\Omega$ be a bounded region in $\mathbb R^n, n\ge 2,$ with smooth boundary
$\partial \Omega$ and $\nu$ denotes the unit outward normal to $\O$.
Then Morgan considered in reference (\cite{Morgan2})
essentially of the form
\begin{subeqnarray}\label{slparabolic}
\bu_t(x,t)&=&D\Delta \bu(x,t) + f(\bu(x,t)), \quad x\in \Omega, t>0,\\
\frac{\partial \bu(x,t)}{\partial \nu}&=&\b0, \quad x\in \partial\Omega, t>0,\\
\bu(x,0) &=& \bu_0(x),\quad x\in \Omega,
\end{subeqnarray}
where $\bu: \Omega\times (0,\infty)\rightarrow \mathbb R^m,$
$f: \mathbb R^m \rightarrow \mathbb R^m$ is a locally Lipschitz continuous function,
$D$ is an $m\times m$ diagonal matrix with diagonal entries $d_j > 0$,
and $\bu_0:\Omega \rightarrow \mathbb R^m$ is bounded and
  measurable.
Then the following theorem holds \cite{Morgan2}:
\begin{theorem}\label{thm1} 
Under the assumptions on \eqref{slparabolic} stated above,
there exists $T_{\max} > 0$ and $M=(M_j)\in C([0,T_{\max}),\mathbb R^m)$ such that \\
{\rm (i)} (\ref{slparabolic}) has a unique classical solution
  $\bu$ on $\overline{\Omega}\times [0,T_{\max})$ which cannot be continued
  to $[0, T)$ for any $T>T_{\max}$,
 and \\
{\rm (ii)} $|\bu_j(\cdot,t)|_{\infty, \Omega}\leq M_j(t)$  for all $1\leq j\leq m, 0\leq t < T_{\max}$. \\
Moreover, if $T_{\max} < \infty$, then $|\bu_j(\cdot,t)|_{\infty, \Omega} \rightarrow \infty$ as $t \rightarrow T_{\max^{-}}$ for some $1\leq j \leq m$.
\end{theorem}
Defining $F_{1}(0,0) = 0 $ and $F_{2}(0,0)=0$ in
our model \eqref{2.3}, \thmref{thm1} implies local existence and uniqueness.
More precisely,
there exists $T_{\max} > 0$ and $N_M$ and $P_M\in C([0,T_{\max}))$ such that \\
{\rm (i)} (\ref{slparabolic}) has a unique classical solution
  $\bu=(N,P)^T$ on $[0,l]\times [0,T_{\max})$ which cannot be continued
  to $[0, T)$ for any $T>T_{\max}$,
 and \\
{\rm (ii)} $|N(\cdot,t)|_{\infty, (0,l)}\leq N_M(t)$ and
$|P(\cdot,t)|_{\infty, (0,l)}\leq P_M(t)$ for
$0\leq t < T_{\max}$. \\
Moreover, if $T_{\max} < \infty$, then either
$|N(\cdot,t)|_{\infty, (0,l)} \rightarrow \infty$ or
$|P(\cdot,t)|_{\infty, (0,l)} \rightarrow \infty$ as $t \rightarrow T_{\max^{-}}.$
\subsection{Turing instability}
\begin{definition}{ We say that the equilibrium $(\overline{N},\overline{P})$ is
Turing unstable if it is an asymptotically stable equilibrium of the
kinetic
system (\ref{kinetic}) but is unstable with respect to solutions of (\ref%
{2.1}) (see \cite{Okubo}).}
\end{definition}
An equilibrium is Turing unstable means that there are solutions of
(\ref{2.3}) that have initial values $\bu(x,0)$
arbitrarily closed to $\overline{\bu}$ (in the supremum norm) but
do not tend to $\overline{\bu}$ as $t$ tends to $\infty$.

We linearize system (\ref{2.1}) at the point $(\overline{N},\overline{P})$:
setting $\bv=(v_{1},v_{2})^T=(N-\overline{N},P-\overline{P})^T,$
the linearized system assumes the form
\begin{equation}
\bv_{t}=A\bv+D\frac{\partial^{2}\bv}{\partial x^{2}},  \label{2.7}
\end{equation}
while the boundary conditions remain unchanged:
\begin{equation}
\bv_{x}(0,t)=\bv_{x}(l,t)=\b0.  \label{2.8}
\end{equation}
The linear boundary value problem \eqref{2.7}-\eqref{2.8}
can be solved in several ways. In particular, the Fourier's method of
separation of variables
assumes that solutions can be represented in the form $\bv(x,t)=\psi(x)\by(t),$
with $\by:[0,\infty )\rightarrow \mathbb R^{2},$ $\psi
:[0,l]\rightarrow \mathbb R.$ Then
\begin{equation}
\frac{d\by}{dt}=(A- \zeta D)\by,  \label{2.9}
\end{equation}
and
\begin{equation}
-\psi_{xx}=\zeta \psi, \quad \psi_x
(0)=\psi_x(l)=0.  \label{2.10}
\end{equation}
The eigenvalues of the boundary value problem (\ref{2.10}) are
\begin{equation}
\zeta_{j}=\left(\frac{j\pi}{l}\right)^{2},\quad j=0,1,2,\cdots  \label{2.11}
\end{equation}
with corresponding eigenfunctions
\begin{equation}
\psi_{j}(x)=\cos \frac{j\pi x}{l}.  \label{2.12}
\end{equation}
Clearly, $0=\zeta_0<\zeta_{1}<\zeta_{2}<\cdots$. These eigenvalues
are to be substituted into (\ref{2.9}).
Denoting by $\by_{1j}$ and $\by_{2j}$ the two linearly independent solutions of (\ref{2.12})
associated with $\zeta =\zeta_{j}$, the
solution of the boundary value problem (\ref{2.7})-(\ref{2.8}) is obtained
in the form
\begin{equation}
\bv(x,t)=\sum_{j=1}^{\infty}(a_{1j}\by_{1j}(t)+a_{2j}\by_{2j}(t))\cos
\frac{j\pi x}{l}  \label{2-13}
\end{equation}
where $a_{ij}, i=1,2, j=0,1,2,\cdots,$ is to be determined according to the
initial condition $\bv(x,0).$
For instance, if $\by_{1j}(0)=(1,0)^T,\by_{2j}(0)=(0,1)^T$ for $j=0,1,2,\cdots$,
\begin{equation*}
\left[\begin{array}{c}
a_{10} \\ a_{20}
\end{array}
\right] =\frac{1}{l}\int\limits_{0}^{l}\bv(x,0)dx,
\end{equation*}
\begin{equation*}
\left[ \begin{array}{c}
a_{1k} \\ a_{2k}
\end{array}
\right] =\frac{2}{l}\int\limits_{0}^{l}\bv(x,0)\cos \frac{k\pi x}{l}\,dx\qquad,k=0,1,2,\cdots.
\end{equation*}
Set
\begin{equation}
B(\zeta )=A-\zeta D,\quad B_{j}=B(\zeta_{j})=A-\zeta_{j}D.
\label{2.14}
\end{equation}
According to Casten and Holland \cite{Casten}, if
both eigenvalues
of $B_{j}$ have negative real parts  for all $j$,
then the equilibrium
$(\overline{N},\overline{P})$ of (\ref{2.3}) is
asymptotically stable; if at least one
eigenvalue of a matrix $B_{j}$ has positive real part, then
$(\overline{N}, \overline{P})$ is unstable.
Recalling (\ref{2.6}), the trace and determinant are given by
\begin{subeqnarray}\label{2.17}
\trace B_{j}&=&\Theta_{1}-\Theta_{4}-\zeta_{j}(d_{1}+d_{2}),
\slabel{2.17a}\\
\det B_{j}&=&\Theta_{2}\Theta_{3}-\Theta_{1}\Theta_{4}
+\zeta_{j}\left\{d_{1}\Theta_{4}-d_{2}\Theta_{1}\right\}
+\zeta_{j}^2d_{1}d_2.
\slabel{2.17b}
\end{subeqnarray}
Notice that \eqref{assume-theta} implies that $\trace{B_j} < 0.$ Therefore
the eigenvalues of $B_j$ have negative real parts if $\det B_j>0$ which
is guaranteed in case
\begin{equation*}
d_{1}\Theta_{4}>d_{2} \Theta_{1},\quad
(d_{1}\Theta_{4}-d_{2} \Theta_{1})^2-4d_{1}d_{2}
(\Theta_{2}\Theta_{3}-\Theta_{1}\Theta_{4})<0.
\end{equation*}
Notice that  $\det B_j<0$ for all sufficiently large $j$ if
$d_1\Theta_4-d_2\Theta_1 < 0,$ since the eigenvalues $\zeta_j$ is monotonic
increasing with its limit $\infty.$
Therefore, one has the following theorem:
\begin{theorem}\label{thm2} Assume that \eqref{assume-gamma-delta} and
\eqref{assume-theta}. Then the equilibrium point
$(\overline{N},\overline{P})$ of \eqref{2.3} is
asymptotically stable if
\begin{equation}
d_{1}\Theta_{4}>d_{2} \Theta_{1},\quad
(d_{1}\Theta_{4}-d_{2} \Theta_{1})^2-4d_{1}d_{2}
(\Theta_{2}\Theta_{3}-\Theta_{1}\Theta_{4})<0;
\end{equation}
while it is Turing unstable if
\begin{equation}
d_1\Theta_4-d_2\Theta_1 < 0 \text{ and }
(d_{1}\Theta_{4}-d_{2} \Theta_{1})^2-4d_{1}d_{2}
(\Theta_{2}\Theta_{3}-\Theta_{1}\Theta_{4})>0,
\label{2.20}
\end{equation}
or if there exist a positive integer $k$ such that
\begin{eqnarray}\label{2.20b}
\det B_{k}&=&\Theta_{2}\Theta_{3}-\Theta_{1}\Theta_{4}
+\zeta_{k}\left\{d_{1}\Theta_{4}-d_{2}\Theta_{1}\right\}
+\zeta_{k}^2d_{1}d_2 < 0.
\end{eqnarray}
\end{theorem}
\subsection{Pattern formation}
For a nonnegative real parameter $\lambda$
consider the reaction-diffusion system to find
$\bu: (0,l)\times (0,\infty) \rightarrow \mathbb R^{n}$ such that
\begin{equation}\label{3.1}
\bu_{t}=\bF(\bu;\lambda )+D(\lambda )\frac{\partial^{2}\bu}{\partial x^{2}},
\end{equation}
where $D$ is a non-negative diagonal matrix
depending smoothly on  $\lambda$ and
$\bF:\mathbb R^{n}\times \lbrack 0,\infty )\rightarrow \mathbb R^{n}$
is a smooth function. Suppose
\eqref{3.1} is equipped with the Neumann boundary condition
\begin{equation}
\bu_{x}(0,t)=\bu_{x}(l,t)=\b0.  \label{3.2}
\end{equation}
Assume further that for some $\overline{\bu}\in \mathbb R^{n}$ we have
$\bF(\overline{\bu};\lambda)=0$ for all $\lambda \in [0,\infty),$ {\it
  i.e.}
$\overline{\bu}$ is a parameter-independent constant stationary solution of
\eqref{3.1}--\eqref{3.2}.

\begin{definition} We say that $\overline{\bu}$ undergoes a Turing
bifurcation at $\lambda_0 \in \lbrack 0,\infty )$ if
the solution $\overline{\bu}$ is asymptotically stable for $0<\lambda <\lambda_{0},$
while it is unstable for $\lambda_{0}<\lambda,$ (or vice versa, {\it i.e.}
the regions for asymptotical stability and instability may be exchanged),
and in some neighborhood of $\lambda_{0}$ the problem
(\ref{3.1})-(\ref{3.2})
has non-constant stationary solution ({\it i.e.} solution which does not
depend on time but depends on space.)
\end{definition}

With $d_1$ fixed, regarding $d_2$ as the parameter $\lambda$,
we will consider the linearized system (\ref{2.7})-(\ref{2.8}) as
a parameter-dependent problem in the setting \eqref{3.1}--\eqref{3.2}.
Notice that $\bu(x,t)=(0,0)^T$ is clearly a solution for (\ref{2.7})-(\ref{2.8}).
Then the condition for a Turing bifurcation for
the linearized system (\ref{2.7})-(\ref{2.8}) is given as follows:

\begin{theorem}\label{thm3} Suppose that $\trace  A<0$  and
$\det A>0.$
\\
{\rm (i)} If
\begin{equation}
d_{1}\geq \frac{\Theta_{1}}{\zeta_{1}},  \label{3.3}
\end{equation}
then the zero solution of the linear problem (\ref{2.7})-(\ref{2.8})
is asymptotically stable for all $d_{2}>0.$
\\
{\rm (ii)} If
\begin{equation}
\frac{\Theta_{1}}{\zeta_{2}} \le d_{1} < \frac{\Theta_{1}}{\zeta_{1}},  \label{3.4}
\end{equation}
then the zero solution of the linear problem (\ref{2.7})-(\ref{2.8})
undergoes a Turing bifurcation at
\begin{equation}
d_{2}:=d_{2crit}=\frac{\Theta_{2}\Theta_{3}-\Theta_{1}\Theta_{4}+\zeta
_{1}d_{1}\Theta_{4}}{\zeta_{1}(\Theta_{1}-\zeta_{1}d_{1})}.
\label{3.5}
\end{equation}
\end{theorem}

\begin{proof} {\rm (i)} Rewriting (\ref{2.17b}) as
\begin{equation*}
\det B_{j}=\Theta_{2}\Theta_{3}-\Theta_{1}\Theta_{4}+\zeta
_{j}d_{1}\Theta_{4}-\zeta_{j}d_{2}(\Theta_{1}-\zeta_{j}d_{1}),
\end{equation*}
we see from \eqref{assume-gamma-delta} and \eqref{assume-theta} that
$\det B_{j}>0$ for all $j=0,1,2,\cdots$ if $d_{1}\geq \Theta_{1}/\zeta_{1}$
holds,
since $\zeta_{j},j=0,1,2,\cdots$ forms a monotone increasing sequence
(\ref{3.3}).
Therefore, the zero solution of (\ref{2.7})-(\ref{2.8}) is asymptotically
stable under such conditions.

{\rm (ii)} Suppose $d_{1}$ satisfies (\ref{3.4}) and choose $\lambda=d_{2}$
as given in \eqref{3.5}. Then $\det B_{1}=0.$ Clearly, we have
$\det B_{1}>0$ for $0<d_{2}<d_{2crit}$, and $\det B_{1}<0$ for
$d_{2crit}<d_{2}$.
In both cases $\det B_{j}>0,j\neq 1.$
Again by Casten and Holland \cite{Casten} as quoted just after formula
(\ref{2.14}),
the zero solution is asymptotically stable for $0<d_{2}<d_{2crit}$,
and it is unstable for $d_{2crit}<d_{2}$.
If $d_{2}=d_{2crit}$, one eigenvalues of $B_{1}$ becomes zero and the other is
$\trace{B_1}$, which is negative.
Denote the eigenvector corresponding to the zero eigenvalue by
$\by_{11}=(\eta_{1},\eta_{2})^T$, {\it i.e.}
\begin{equation*}
B_{1}\by_{11}=(A-\zeta_{1}D)\by_{11}=0,\quad \by_{11}\neq \b0.
\end{equation*}
As we can see from (\ref{2.9})-(\ref{2.12}) the function
\begin{equation*}
\bv_{1}(x,t):=\by_{11}\psi_{1}(x)=\left[
\begin{array}{c}
\eta_{1} \\
\eta_{2}
\end{array}
\right] \cos \frac{\pi x}{l},
\end{equation*}
is a spatially non-constant stationary solution of the linearized problem
(\ref{2.7})-(\ref{2.8}).
This implies that the zero solution undergoes Turing
bifurcation
at $d_{2crit}$. This completes the proof.
\end{proof}

In the remaining part of this section we will extend the latter
result about the Turing bifurcation of the zero solution of the
linearized system to the non-linear problem (\ref{2.3}). For this
we need the following:

\begin{theorem}\label{thm4} Let $X$ and $Y$ be Banach spaces,
$U=V\times S$  an open subset of $X \times \mathbb R$, and $\mathbf f\in C^{2}(U;Y)$%
 such that $\mathbf f(\b0,\lambda)=\b0,\lambda \in S\subset \mathbb R.$
Denote by $L_{10}=\mathbf f_{\bv}(\b0,\lambda_{0})$ and $L_{12}=\mathbf f_{\bv,\lambda}(\b0,\lambda_{0})$
the linear operators obtained by differentiating $\mathbf f$ with
respect to its first variable only and the first and second variables at
$\bv=\b0\in V, \lambda_{0}\in S$, respectively. Assume
that the following conditions hold:

{\rm (i)} the kernel of $L_{10}$, the subspace $N(L_{10})$
of $X$ is a one-dimensional vector space spanned by $\bv_{1}\in X$;

{\rm (ii)} the range of $L_{10}$, the subspace $R(L_{10})$
 of $Y$ has codimension $1$, {\it i.e.}
$dim[Y/R(L_{10})]=1$;

{\rm (iii)} $L_{12}\bv_{1}\notin R(L_{10})$.

Let $Z$ be an arbitrary closed subspace of $X$
such that $X=[Span$ $\bv_{1}]\oplus Z;$ then there is a
$\delta >0$ and $C^{1}$-curve 
$(\phi, \lambda):(-\delta ,\delta)\rightarrow Z\times S$ such that;
$\phi (\b0)=\b0;$ $\lambda (0)=\lambda_{0};$
$\mathbf f(s\bv_{1}+s\phi (s), \lambda(s))=\b0$ for $|s| <\delta.$
Furthermore, there is a neighborhood of $(\b0,\lambda_{0})$
such that any zero of $\mathbf f$ either lies on this curve or
is of the form $(\b0,\lambda_{0}).$
\end{theorem}

\begin{proof} The idea of the proof is to introduce a new parameter $s$
which enables to apply immediately the implicit function theorem for the
function $\bF\in C^{1}(U\times Z,Y)$ defined by
\begin{equation*}
\bF(\lambda,s,z):=
\begin{cases}
\frac1{s} \mathbf f(s\bv_{1}+s\mathbf z,\lambda) & \text{ if }s\neq 0, \\
L_{10}(\bv_{1}+\mathbf z), & \text{ if }s=0.
\end{cases}
\end{equation*}%
See, for the details of the proof of the theorem,  pp. 172--173 of \cite{Smoller}.
\end{proof}

\begin{remark} In what follows the role of the space $X$ will be played by
\begin{equation}
X=\left\{\bV\in C^{2}([0,l];\mathbb R^{2}):\text{ }\bV_{x}(0)=\bV_{x}(l)=\b0\right\}  \label{3.7}
\end{equation}
with the norm
$\left\Vert \mathbf f\right\Vert_{X}=\sum_{0\leq \alpha \leq 2}
\sup_{x\in [0,l]} | \partial^{\alpha }\mathbf f(x)|,$
where $|\cdot|$ denotes the
usual vector or matrix-norm, while $Y=C^{0}([0,l],\mathbb R^{2})$ with
the norm $\| \mathbf f\|_{Y}=\sup_{x\in [0,l]}| \mathbf f(x)|.$ However, in choosing the subspace $Z$
of $X$ we shall use the orthogonality induced by the inner product
\begin{equation*}\label{3.8}
\left\langle \bv,\bw\right\rangle
=\int\limits_{0}^{l}\left[v_{1}(x)w_{1}(x)+v_{2}(x)w_{2}(x)\right]dx,\quad\text{for
}\bv=(v_{1},v_{2})^T, \bw=(w_{1},w_{2})^T.
\end{equation*}
\end{remark}

\begin{theorem}\label{thm5} Suppose that $\trace A<0$ and $\det A>0.$

{\rm (i)} If (\ref{3.3}) holds, then the constant solution $\overline{\bu}%
=(\overline{N},\overline{P})^T$ of the nonlinear problem (\ref{2.3})
is asymptotically stable.

{\rm (ii)} If $(0,\eta_{2})^T$ is not parallel to the second eigenvector 
$\by_{21}$ of $B_{1}$ and $d_{1}$
satisfies (\ref{3.4}), then at $d_{2}=d_{2crit}$ the constant
solution $\overline{\bu}$ undergoes a Turing bifurcation.
\end{theorem}

\begin{proof} {\rm (i)} follows immediately from the asymptotic stability of
the zero solution of the linear problem \textit{\ }(\ref{2.7})-(\ref{2.8})%
\textit{.}

{\rm (ii)} As in the proof of {\rm (i)} of \thmref{thm3}, we have that
 $\overline{\bu}$ is asymptotically stable for $d_{2}\in (0,d_{2crit})$, while
 it is unstable  for $d_{2}\in (d_{2crit},\infty )$.
We have to show the existence of a
stationary non-constant solution in some neighborhood of the critical value
$d_{2crit}.$ Such a stationary solution $\overline{\bu}$ satisfies
the following system of second-order partial differential equations
\begin{equation}
D\bu_{xx}+\bF(\bu)=0,\quad \bu_{x}(0)=\bu_{x}(l)=\b0.  \label{3.9}
\end{equation}
We consider (\ref{3.9}) as an operator equation on the Banach space $X$
given by (\ref{3.7}), and we apply \thmref{thm3} with $d_{2}$ as the bifurcation
parameter.
Set
$\bv:=\bu-\overline{\bu}$. Then (\ref{3.9}) assumes the equivalent form
\begin{equation}
D\bv_{xx}+A\bv+\bG(\bv)=0,\text{ \ }\bv_{x}(0)=\bv_{x}(l)=\b0.  \label{3.10}
\end{equation}
where $A$ is the Jacobian matrix of $\bF$ evaluated at $\overline{\bu}$ and
\begin{equation}
\bG(\bv)=\bF(\bv+\overline{\bu})-A\bv,\quad \bG(\b0)=\b0,\bG_{\bv}(\b0)=\b0.  \label{3-11}
\end{equation}
Denote the left hand side of (\ref{3.10}) by $T(\bv, d_{2}),$ where
$T$ is a
one-parameter family of operators acting on $X$ and taking its elements into
$Y=C^{0}([0,l];\mathbb R^{2}).$
Clearly, $T$ is a $C^{2}$ mapping.
The spectrum of the linear operator
$L_{10}=T_{\bv}(\b0,d_{2crit})=\frac{\partial T}{\partial \bv}(\b0, d_{2crit})$
consists of the eigenvalues $\mu_{ij}$  of the matrices
$B_{j}$ given by (\ref{2.14}) with its corresponding eigenfunctions are
$\psi_{j}(x)y_{ij},$ $i=1,2, j=0,1,2,\cdots,$
where $\psi_{j}$ is given by (\ref{2.12})
and $\by_{ij}$ is the eigenvector of the matrix $B_{j}$ corresponding to the
eigenvalues $\mu_{ij}$ (see \eqref{2-13}). Now, all
matrices $B_{j}=A-\zeta_{j}D$ are to be taken at $d_{2}=d_{2crit}$.
As it can be seen from the proof of \thmref{thm3} and from \eqref{2.17}
for $i=1,2;$ for all nonnegative integer $j$ except $j=1,$ all $\mu_{ij}$
have negative real parts. For $j=1$ one eigenvalue $\mu_{11}$ is equal to $0$
and the other $\mu_{21}$ is negative. The eigenfunction corresponding to $\mu_{11}=0$ is
$\bv_{1}=\by_{11}\cos (\pi x/l).$ Thus, the null-space of the operator
$L_{10}=T_{\bv}(\b0,d_{2crit})$ is a one-dimensional linear space spanned by
$\bv_{1}.$
Owing to the orthogonality and completeness of the eigenfunction system of the operator
$-\frac{\partial^2}{\partial x^2}$,  the range of this operator is given by
\begin{eqnarray*}
R(L_{10}) &=&\big\{\bw\in C^{0}([0,l];\mathbb R^{2}):\text{ the eigenfunction expansion of}
\\
&&\qquad \bw\text{ does not contain } \cos \frac{\pi x}{l}    \big\} \cup
\span\big\{\by_{21} \cos \frac{\pi x}{l} \big\},
\end{eqnarray*}%
so that the codimension of $R(L_{10})$ is one.

Let $L_{12}=\frac{\partial T_{\bv}}{\partial d_{2}} (0, d_{2crit}).$ Then
\begin{equation*}
L_{12}=D^{^{\prime }}\frac{\partial^{2}}{\partial x^{2}}\text{ \ \ where \ }%
D^{^{\prime }}=\frac{\partial D}{\partial d_{2}}=\left[
\begin{array}{cc}
0 & 0 \\
0 & 1%
\end{array}%
\right] .
\end{equation*}%
Clearly,
\begin{equation*}
L_{12}\bv_{1}=-\(\frac{\pi }{l}\)^{2}\cos \frac{\pi x}{l}  D^{^{\prime }} \by_{11}
=-\(\frac{\pi }{l}\)^{2}  \cos \frac{\pi x}{l} \left[
\begin{array}{c}
0 \\
\eta_{2}%
\end{array}%
\right].
\end{equation*}%
Under the assumption $L_{12}\bv_{1}\nparallel $ $\by_{21}\cos \frac{\pi x}{l}$,
we see that $L_{12}\bv_{1}$ does not belong to $R(L_{10}),$ fulfilling
the condition {\rm (iii)} of \thmref{thm4}.

Letting
\begin{equation*}
Z=R(L_{10}),
\end{equation*}
which is a closed subspace of $Y$,
we verify that all the hypotheses of \thmref{thm4} hold; moreover,
$(\b0, d_{2crit})$ is a bifurcation point, and there exist a $\delta >0,$ a
function
$d_{2}:(-\delta ,\delta )\rightarrow \mathbb R$ such that for
$s\in (-\delta ,\delta )$
\begin{equation*}
\bv(x;s)=s\by_{11}\cos \frac{\pi x}{l}+s\phi (x;s)
\end{equation*}
is a solution of (\ref{3.10}) with
$d_{2}=d_{2}(s)$, $\left| s\right| <\delta$,
$d_{2}(0)=0,\phi (x;0)=\b0,$ and $d_{2}\in C^{1},\phi (x;\cdot)\in C^{1},\phi
(\cdot;s)\in Z.$
\end{proof}
\begin{remark} The corresponding solution of (\ref{3.9}), {\it i.e.} the
non-constant stationary solution of the nonlinear parabolic system (\ref{2.3})
is
\begin{equation}\label{pattern}
\bu(x;s)=\overline{\bu}+s\mathbf y_{11}\cos \frac{\pi x}{l}+O(s^{2}),
\end{equation}
(corresponding to the choice $d_{2}=d_{2}(s),\left\vert s\right\vert <\delta
$), {\it i.e.}
\begin{subeqnarray}\label{3.25}
N(x) &=&\overline{N}+s\eta_{1}\cos \frac{\pi x}{l}+O(s^{2}), \\
P(x) &=&\overline{P}+s\eta_{2}\cos \frac{\pi x}{l}+O(s^{2}).
\end{subeqnarray}%
since $s$ is considered to be small here, this solution is called as a
\textit{small amplitude pattern.}
\end{remark}

\begin{remark} Because of \thmref{thm4}
\textit{(\ref{2.3})} has
no other stationary solution apart from $(\overline{N},\overline{P})$ and 
\eqref{pattern} in a neighborhood of $(\overline{\bu}, d_{2crit})\in \mathbb R\times X.$%
\end{remark}

\begin{remark} In the linear case (by \thmref{thm3}) for the function $d_{2}$%
\ holds: $d_{2}(s)=d_{2crit},$ and a corresponding one parameter family of
solutions is $\overline{\bu}+s\bv_{1}, s\in \mathbb R.$
\end{remark}

\section{Numerical approximation}
\subsection{The numerical scheme}
The reaction-diffusion equations (\ref{2.3}) are solved numerically
using the forward Euler method in time, the centered difference method in
space. This numerical scheme gives a stable solution under a certain 
that stasisfies the CFL (Courant-Friedrichs-Lewy) condition.
The details are as follows.

Consider the computational domain $[0,1]$ and the mesh size $h$
and the time step size $\Delta t,$ which will be determined later in \eqref{eq:Deltat}.
Set $N_h = \frac{1}{h}.$ Denote by $N_j^k$ and $P_j^k$
the numerical approximation of $N(jh,k\Delta t)$,
$P(jh,k\Delta t)$, respectively for $j=0,1,\cdots,N_h$ and $k=1,2,\cdots.$ 
Then, given initial data $N_j^0, P_j^0, j =
0, 1,\cdots, N_h,$ the numerical scheme is to solve
\begin{subeqnarray}\label{eq:NjPj}
N_j^{k+1}&=&N_j^k + {\Delta t} N_j^k \left[1-N_j^k - \frac{\alpha
    P_j^k}{P_j^k + N_j^k}\right] + \Delta t d_1\frac{N_{j-1}^k - 2N_j^k + N_{j+1}^k}{h^2},\\
P_j^{k+1}&=&P_j^k + {\Delta t}\epsilon P_j^k
\left[-\frac{\gamma+\delta\beta P_j^k}{1 + \beta P_j^k} + 
  \frac{N_j^k}{P_j^k + N_j^k}\right] + \Delta t d_2\frac{P_{j-1}^k - 2P_j^k + P_{j+1}^k}{h^2}
\end{subeqnarray}
for $j = 1, 2,\cdots, N_h-1,$ iteratively for $k =1, 2, \cdots.$
On the boundaries $x=0, x=1$ where Neumann condition holds, 
we used a three-point interpolation scheme to guarantee the second-order
accuracy in space as follows:
\begin{subeqnarray}
\label{eq:bc}
N_{2}^k - 4 N_{1}^k + 3 N_{0}^k = 0;&\quad& P_{2}^k - 4 P_{1}^k + 3 P_{0}^k = 0; \\
N_{N_h-2}^k - 4 N_{N_h-1}^k + 3 N_{N_h}^k = 0;&\quad& P_{N_h-2}^k - 4 P_{N_h-1}^k + 3
P_{N_h}^k = 0.
\end{subeqnarray}
We will then establish the
the positivity of the numerical solutions and boundedness 
for the numerical prey solution under certain conditions on $\Delta t$.
Suppose that $0 \leq N_j^k \leq 1$ for $j=1,\cdots, N_h-1.$ 
Then, for $j=2,\cdots,N_h-2,$
\begin{eqnarray}
N_j^{k+1}&=&N_j^{k} + {\Delta t} N_j^k \left[1-N_j^k - \frac{\alpha
    P_j^k}{P_j^k + N_j^k}\right] + \Delta t d_1\frac{N_{j-1}^k - 2N_j^k +
  N_{j+1}^k}{h^2} \nonumber \\
&\leq&  N_j^k + {\Delta t} N_j^k \left[1-N_j^k \right] + 2\frac{\Delta t
    d_1}{h^2}(1- N_j^k)\nonumber  \\
&=& N_j^k + {\Delta t} (1-N_j^k) \left[N_j^k +\frac{2d_1}{h^2}\right] \nonumber \\
&\leq& N_j^k + {\Delta t} (1-N_j^k) (1+\frac{2d_1}{h^2}) \nonumber \\
&=&\left[1-\Delta t(1+\frac{2d_1}{h^2})\right]N_j^{k} + \Delta
    t(1+\frac{2d_1}{h^2})\nonumber  \\
&\leq& 1-\Delta t(1+\frac{2d_1}{h^2}) +\Delta t(1+\frac{2d_1}{h^2}) \leq 1
\label{est:Nj}
\end{eqnarray}
provided  $1-\Delta t(1+\frac{2d_1}{h^2}) \geq 0.$
For $j=1$, owing to the boundary condition \eqref{eq:bc}, 
$N_1^{k+1}$ is given by
\begin{eqnarray}\label{est:N1}
N_1^{k+1}&=&N_1^{k} + {\Delta t} N_1^k \left[1-N_1^k - \frac{\alpha
    P_1^k}{P_1^k + N_1^k}\right] + \Delta t d_1\frac{ 4 (- N_{1}^k + N_2^k)}{3h^2}.
\end{eqnarray}
Hence, the same analysis as above yields, instead of \eqref{est:Nj},
\begin{eqnarray*}
N_1^{k+1} \le \left[1-\Delta t(1+\frac{4d_1}{3h^2})\right] N_1^k +\Delta t (1+\frac{4d_1}{3h^2}) \le 1
\end{eqnarray*}
provided  $1-\Delta t(1+\frac{4d_1}{3h^2}) \geq 0.$
Analgously, one gets
\begin{eqnarray*}
N_{N_h-1}^{k+1}&=&N_{N_h-1}^{k} + {\Delta t} N_{N_h-1}^k \left[1-N_{N_h-1}^k - \frac{\alpha
    P_{N_h-1}^k}{P_{N_h-1}^k + N_{N_h-1}^k}\right] + \Delta t d_1\frac{ 4 (- N_{N_h-1}^k + N_{N_h-2}^k)}{3h^2},
\end{eqnarray*}
and therefore
\begin{eqnarray*}
N_{N_h-1}^{k+1} \le \left[ 1-\Delta t(1+\frac{4d_1}{3h^2})\right] N_{N_h-1}^k +\Delta t (1+\frac{4d_1}{3h^2}) \le 1
\end{eqnarray*}
provided  $1-\Delta t(1+\frac{4d_1}{3h^2}) \geq 0.$
On the other hand, suppose that $0\le N_j^k\le 1$  for $j=1,\cdots, N_h-1.$ 
Then, for $j = 2,\cdots, N_h-2,$
\begin{eqnarray}\label{est:Nj0}
N_j^{k+1}&=&N_j^{k} + {\Delta t} N_j^k \left[1-N_j^k - \frac{\alpha
    P_j^k}{P_j^k + N_j^k}\right] + \Delta t d_1\frac{N_{j-1}^k - 2N_j^k +
  N_{j+1}^k}{h^2}\nonumber \\
&\geq& N_j^k + \Delta t N_j^k(1-N_j^k) - \Delta t \alpha N_j^k -2\frac{\Delta t
  d_1}{h^2}N_j^k\nonumber\\
&\ge&(1+\Delta t-\Delta t\alpha)N_j^k-\Delta t N_j^k - 2\frac{\Delta t
  d_1}{h^2}N_j^k\nonumber\\
&=&\left[1-\Delta t \alpha-\Delta t\frac{2d_1}{h^2}\right]N_j^{k}
 \geq  0,
\end{eqnarray}
provided  $1-\Delta t \alpha-\Delta t\frac{2d_1}{h^2} \geq 0.$
Next for $j=1$, by using \eqref{est:N1}, the procedure to get the
estimate \eqref{est:Nj0} leads to
\begin{eqnarray}
N_1^{k+1}\ge \left[1-\Delta t \alpha-\Delta t\frac{4d_1}{3h^2}\right]N_1^{k}
 \geq  0
\end{eqnarray}
provided
$1-\Delta t \alpha-\Delta t\frac{4d_1}{3h^2} \geq 0.$ Similarly, under
the same conditions, one obtains
\begin{eqnarray}
N_{N_h-1}^{k+1}\ge \left[1-\Delta t \alpha-\Delta t\frac{4d_1}{3h^2}\right]N_{N_h-1}^{k}
 \geq  0.
\end{eqnarray}

Next, suppose that $P_j^k \geq 0$ and $0 \leq N_j^k\leq 1$  for
$j=1,\cdots,N_h-1.$ Recalling \eqref{assume-gamma-delta}, 
one then obtains, for $j=2,\cdots,N_h-2,$
\begin{subeqnarray}
P_j^{k+1}&=&P_j^k + {\Delta t}\epsilon P_j^k
\left[-\delta +\frac{\delta-\gamma}{1 + \beta P_j^k} + 
  \frac{N_j^k}{P_j^k + N_j^k}\right] + \Delta t d_2\frac{P_{j-1}^k - 2P_j^k +
  P_{j+1}^k}{h^2} \nonumber \\
&\geq&P_j^k -  {\Delta t}\epsilon \delta P_j^k
 - \frac{2\Delta t d_2}{h^2} P_j^k \nonumber\\
&\geq& (1-\epsilon \delta \Delta t - \frac{2\Delta t\, d_2}{h^2})P_j^k
\geq 0 \label{est:Pj}
\end{subeqnarray}
provided $1-\epsilon\delta \Delta t - \frac{2\Delta t\, d_2}{h^2}\geq 0$.
For $j=1$ and $j=N_h-1$, taking into account of
the boundary condition \eqref{eq:bc}, one gets
\begin{subeqnarray}
P_j^{k+1}
\geq (1-\epsilon \delta \Delta t - \frac{4\Delta t\, d_2}{3h^2})P_j^k
\geq 0 \quad{for } j = 1\text{ and }j = N_h-1
\end{subeqnarray}
provided $1-\epsilon\delta \Delta t - \frac{4\Delta t\, d_2}{3h^2}\geq 0$.
Collecting all the above results, we are now in a position to state the
following theorem:
\begin{theorem} Let $0\le N_j^0\le 1, 0\le P_j^0$ for $j=0,\cdots,N_h.$
Suppose that
\begin{eqnarray}\label{eq:Deltat}
\Delta t \leq \min\left(\frac{h^2}{\alpha h^2 +2d_1},\frac{h^2}{ h^2
    +2d_1},\frac{h^2}{\epsilon \delta h^2 +2d_2}\right).
\end{eqnarray}
Then the numerical solutions $N_j^k$ and $P_j^k$ obtained iteratively by
\eqref{eq:NjPj} and \eqref{eq:bc} satisfies that
\begin{eqnarray}
0\le N_j^k\le 1,\quad 0\le P_j^k, \quad\text{for }j=1,\cdots,N_h-1,
 \quad\text{for }k=0,1,2,\cdots.
\end{eqnarray}
\end{theorem}
Numerically a steady state is declared to reach when either the $L_2$ or
$L_{max}$-norm difference is less than a given tolerance value. The $L_2$ and
$L_{max}$-norm differences are defined as follows:
\begin{eqnarray*}
\left\Vert \bu(\cdot,k\Delta
  t)\right\Vert^{2}_{2}&=&\int_0^1\left|(\bu_{steady}(x,k\Delta
  t)-\bu_{h}(x,k\Delta t)\right|^2dx,\\
\left\Vert \bu(\cdot,k\Delta t)\right\Vert_{\infty}&=&\underset{x\in[0,l]}{\max
}\left|\bu_{steady}(x,k\Delta t)-\bu_{h}(x,k\Delta t)\right|,
\end{eqnarray*}
where $\bu_{steady}$ are given by \eqref{3.25} with $O(s^2)$ terms neglected
and $\bu_{h}(x,k\Delta t)$ is the piecewise linear interpolation of
the numerical solution $(N_j^k,P_j^k), j = 0, \cdots, N_h.$ 

\subsection{Numerical examples}
Set $\epsilon =1,\alpha =1.1,\gamma =0.05,\beta
=1,\delta =0.5$. The unique positive equilibrium is $(\overline{N},%
\overline{P})=(0.113585,0.471397).$ If we fix $l=1$ for the length
of the habitat the interval (\ref{3.4}) becomes
\begin{equation*}
1.488790091\times 10^{-3} \leq d_{1} < 5.95160365\times 10^{-3}.
\end{equation*}
In the following \figref{fig:d1d2}, stability regions, the mean
prey-predator diffusion coefficients, $d_{1}$ and $d_{2}$, are
plotted.

We tested our model in the cases of $(d_{1},d_{2})$ =
(0.005,0.2) and $(d_{1},d_{2})$ = (0.005,0.32), which are in the
stable and unstable regions with varying $s= 0.05, 0.1,0.2,0.3,0.4$, respectively. In these
cases, the critical value for Turing bifurcation $d_c$ is $0.271$.
\Fig{fig:merge_NP} shows the numerical prey and predator
solutions, $N$ and $P$, with respect to time at a specified fixed point
$x=0.25$. As shown in \figref{fig:merge_NP}, for $(d_{1},d_{2})$
=(0.005,0.2), the equilibrium solution
$(\overline{N},\overline{P})$ is asymptotically stable and for
$(d_{1},d_{2})=(0.005,0.32)$, the equilibrium solution
$(\overline{N},\overline{P})$ is unstable. 
For the simulation in the case of $(d_{1},d_{2}) = (0.005,0.2),$ 
we used the spatial mesh size $h=0.005,$ and the time step size $\Delta t=
0.00006$ 
determined by the \eqref{eq:Deltat}.
The iteration was run until the time equals to 1000,
with approximately $1.6\cdot 10^7$ iterations. In the case of $(d_{1},d_{2}) =
(0.005,0.32),$ 
the mesh size $h$=0.005 and the time step size $\Delta t$= 0.0000375 were 
used, which were alsothe \eqref{eq:Deltat}.
In this case also the simulation was done until the time equals to 1000,
with approximately $2.6\cdot10^7$ iterations.
In \figref{fig:d202},
in case of $(d_{1},d_{2}) = (0.005,0.2),$ the prey
and predator solutions are plotted with respect to number of iterations and
space. 
We clearly see that as time goes to infinity, the solution converges to the
equilibrium solution $(\overline{N},\overline{P})$. In the lower figure in
\figref{fig:d202}, in case of $(d_{1},d_{2}) = (0.005,0.32),$ where $d_2$ is
in unstable region, the prey
and predator solutions are plotted with respect to number of iterations and
space. We clearly see that as time goes to infinity, the solution shows the
deviation from the equilibrium solution $(\overline{N},\overline{P})$. 

In \figref{fig:varying-s},  for the values near $d_c$,
$(d_{1},d_{2})$ = (0.005,0.27) and $(d_{1},d_{2})$ = (0.005,0.272)
are considered.  By varying $s$ values from 0.05 to 0.4, the prey predator
solution has a \textit{small amplitude pattern} which we expected in the
theory.
In \figref{fig:merge-as-preda} and \figref{fig:merge-as-prey}, we have plotted 
the prey and predator solutions and their small amplitude patterns with
respect to number of iterations and space by changing $s$ values.
Near the $d_c$, in case of $(d_{1},d_{2}) = (0.005,0.27),$ we use the mesh sizes $h=0.005,
\Delta t= 0.0000444$ and ran our simulation until the number of iteration is approximately $10^7$.  
In case of $(d_{1},d_{2}) = (0.005,0.272),$ we have used 
with the mesh sizes $h=0.005$ and $\Delta t= 0.0000441176$.
Again our runs were continued until the number of iteration was approximately $10^7$.
In \figref{fig:merge-as-preda} and \figref{fig:merge-as-prey}, the axis scale in $s=0.1$ has been used as
that of the case of $s=0.4$ which has a bigger amplitude pattern.
Comparing the solutions in \figref{fig:merge-as-preda} and
\figref{fig:merge-as-prey} with the non-constant
stationary solution (\ref{3.25}), we clearly observe that as time goes
to infinity the prey and predator solutions converge to
non-constant stationary solution (\ref{3.25}) which confirms that
$(\overline{N},\overline{P})$ undergoes a Turing bifurcation.

\section*{Discussions}
System (\ref{2.1}) describes the dynamics of a ratio-dependent
predator-prey interaction with diffusion. Prey quantity grows
logistically in the absence of predation, predator mortality is
neither a constant nor an unbounded function, but it is increasing
with the predator abundance and both species are subject to
Fickian diffusion in a one-dimensional spatial habitat from which
and into which there is no migration. It is assumed that the
system without diffusion has a positive equilibrium and under certain
conditions it is asymptotically stable. We show that analytically
at a certain critical value a diffusion driven (Turing type)
instability occurs, {\it i.e.} the stationary solution stays
stable with respect to the kinetic system (the system without
diffusion). We also show that the stationary solution becomes
unstable with respect to the system with diffusion and that Turing
bifurcation takes place: a spatially non-homogenous (non-constant)
solution (structure or pattern) arises. A first order
approximation of this pattern \eqref{pattern} is explicitly given.
A numerical scheme that preserve the positivity of the numerical solutions
and the boundedness of prey solution is introduced.
Numerical examples are also included.

\section*{Acknowledgments}
Research partially supported by the BK21 Mathematical Sciences
Division, Seoul National University,  KOSEF (ABRL) R14-2003-019-01002-0, and KRF-2007-C00031.




\newpage
\begin{figure}[htb]
\begin{center}
\epsfig{file=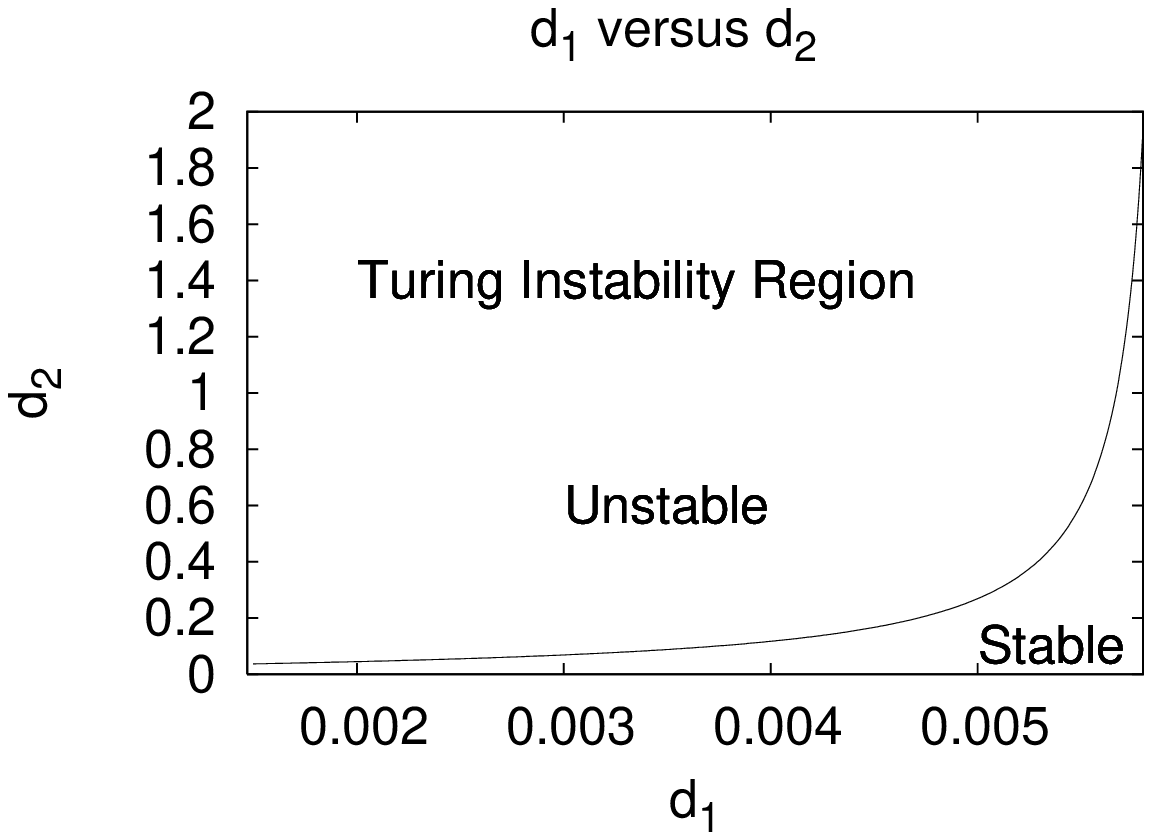, width=12cm, height=9cm}
\caption{\label{fig:d1d2}
$d_1$ and $d_2$ plot, from equation (\ref{3.5})
}
\end{center}
\end{figure}
\newpage
\begin{figure}[hbt]
\centering
{
\includegraphics[width=5cm, height=5cm]{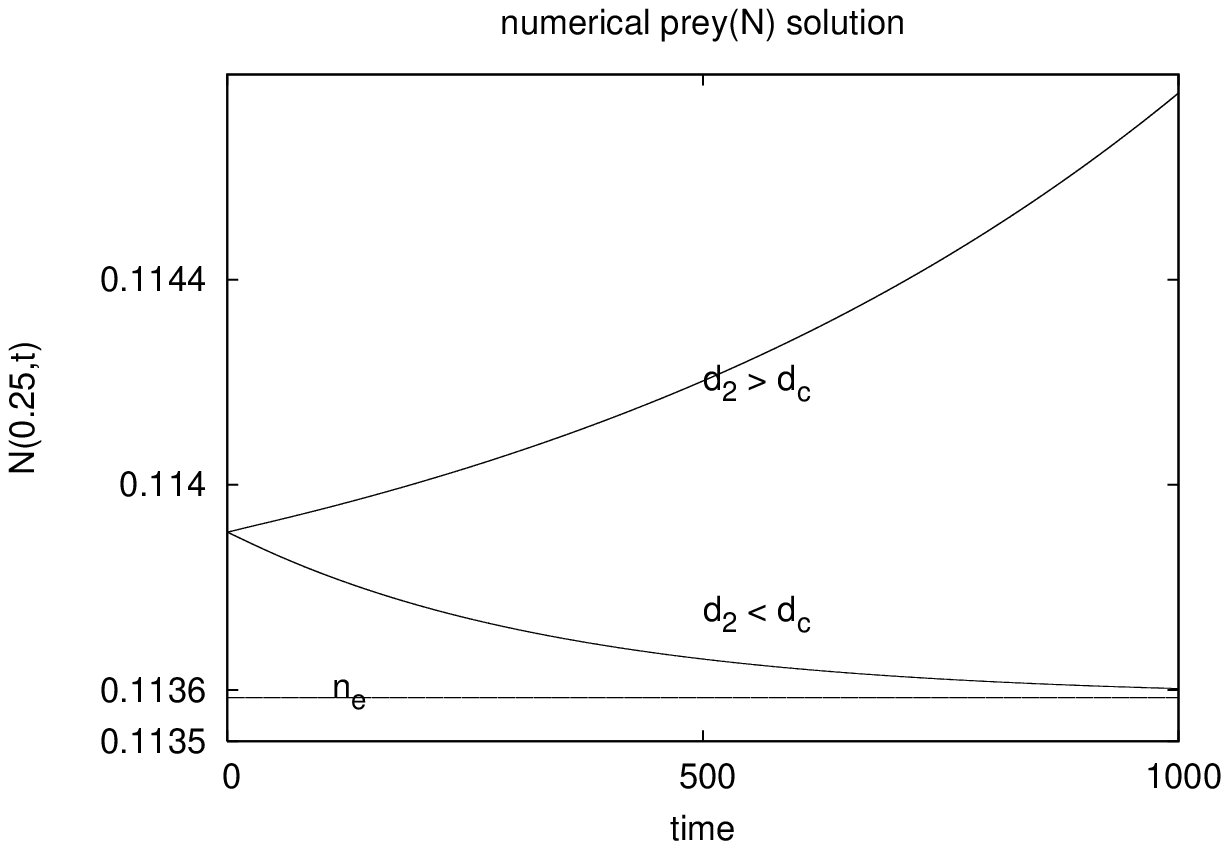}
\hskip 0.3cm
\includegraphics[width=5cm, height=5cm]{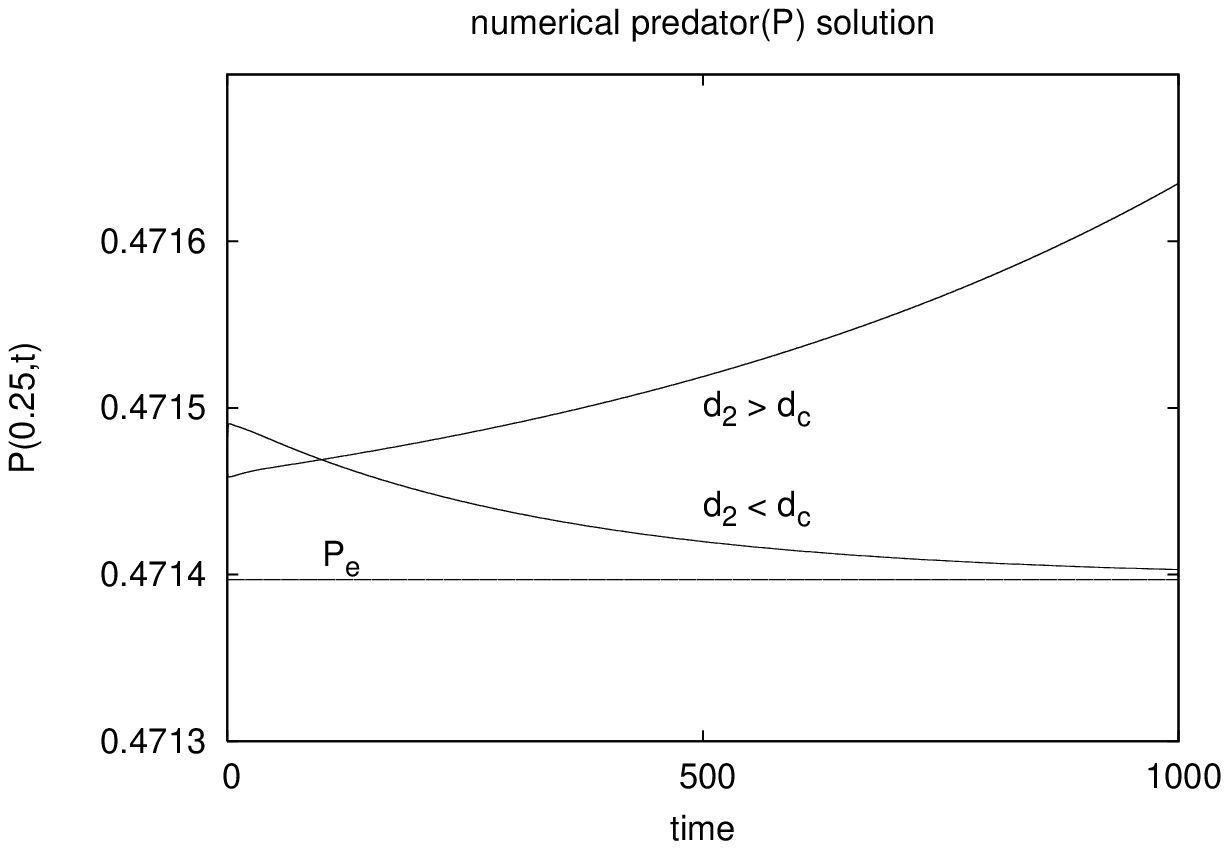}
}\\
\caption{\label{fig:merge_NP} Left: The prey solution at $x$=0.25
with respect to time, the constant line represents $N_e=\overline{N}$ and the two
solid lines represent two different $d_2$ values. Right: The
predator solution at $x$=0.25 with respect to time, the constant
line represents $P_e=\overline{P}$ and the two solid lines represent two different
$d_2$ values.}
\end{figure}
\newpage
\begin{figure}[hbt]
\centering
{
\includegraphics[width=6cm, height=6cm]{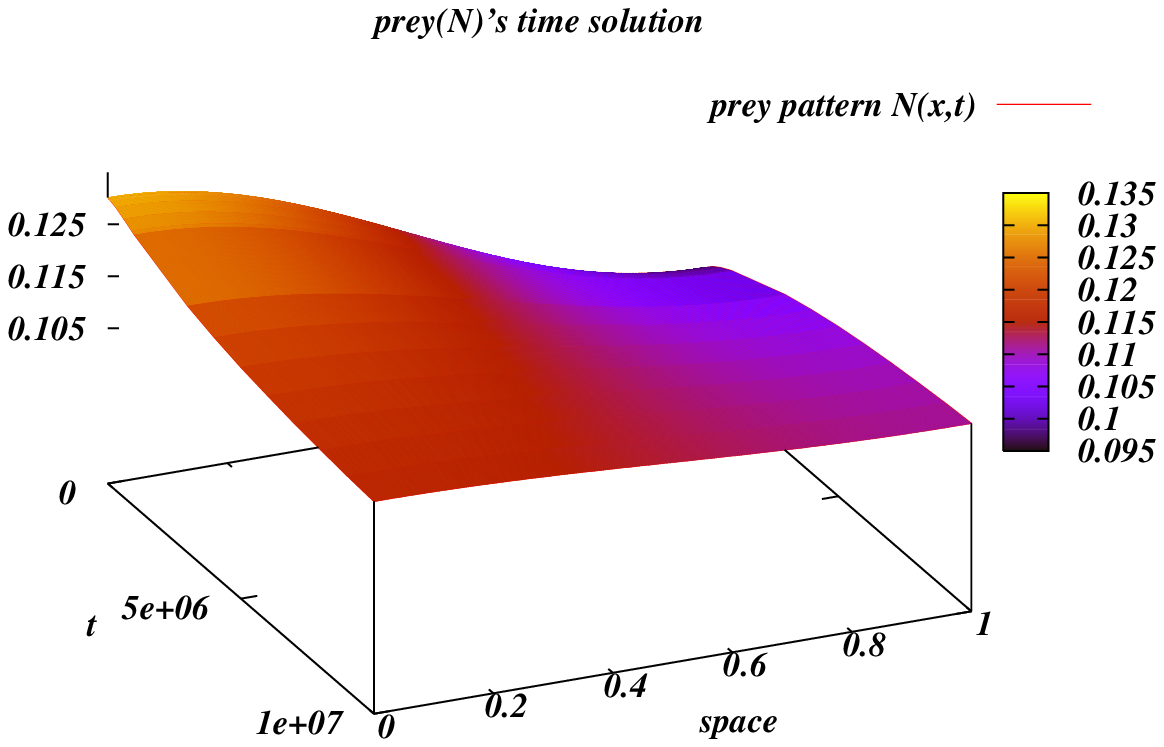}
\hskip 0.3cm
\includegraphics[width=6cm, height=6cm]{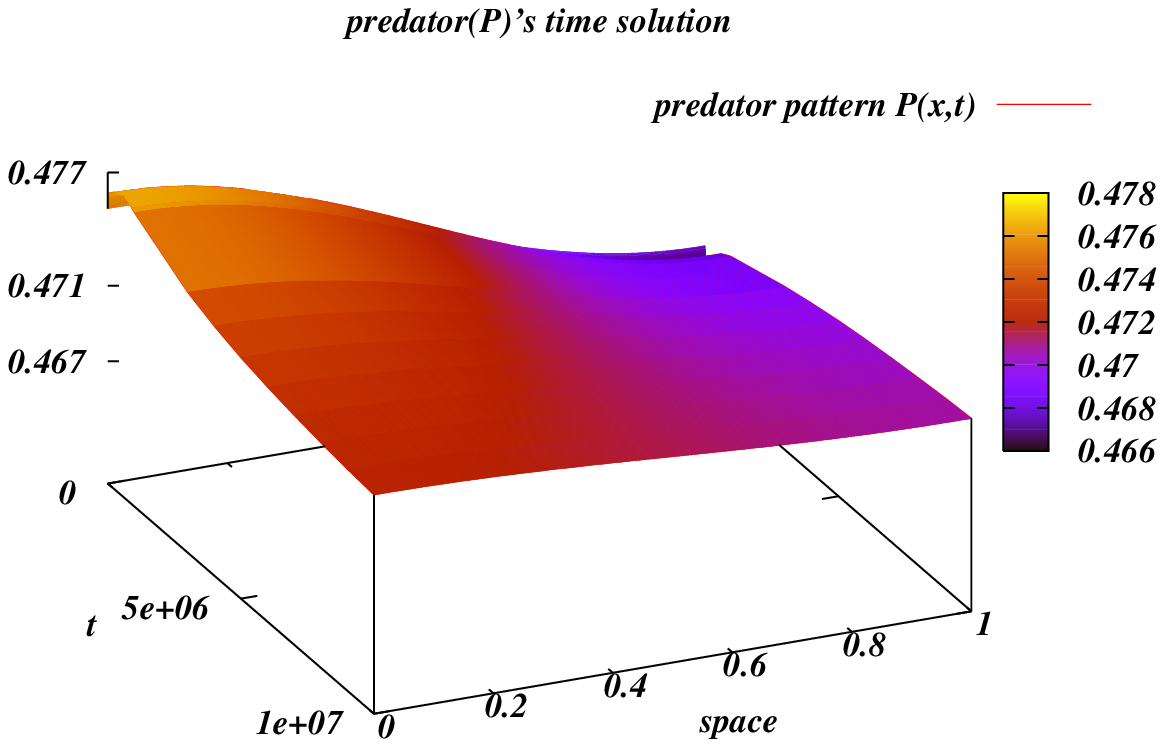}
\vskip 0.3cm
\includegraphics[width=6cm, height=6cm]{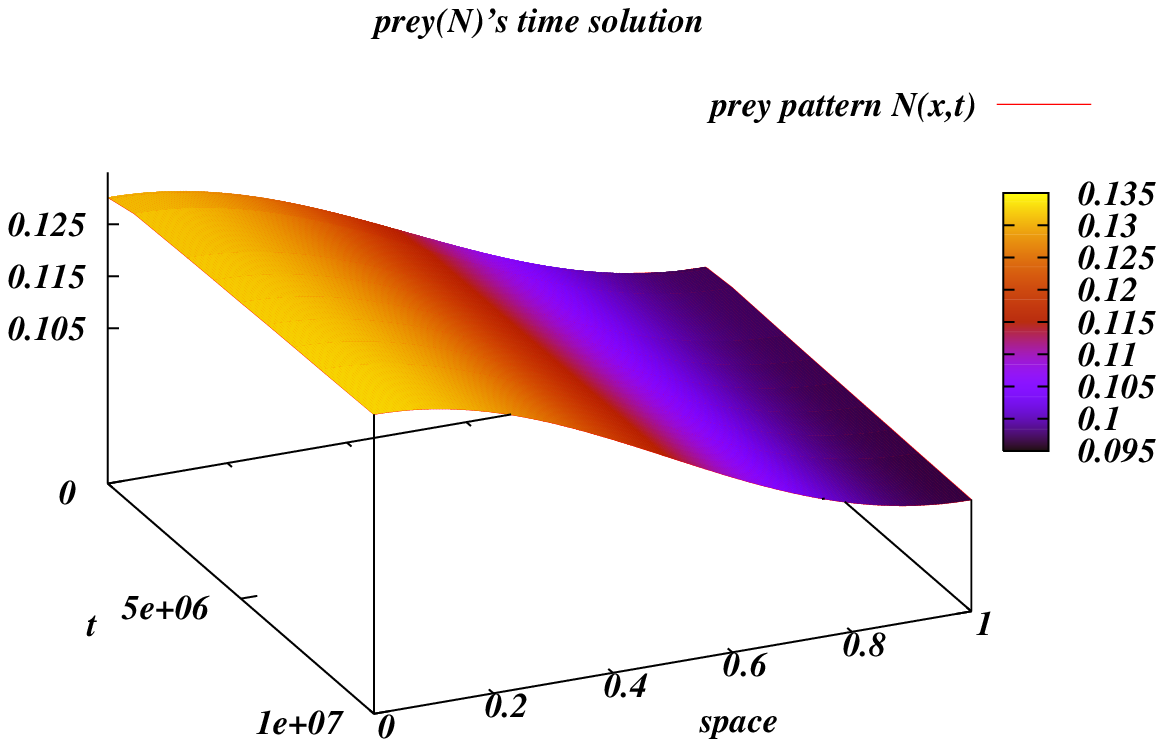}
\hskip 0.3cm
\includegraphics[width=6cm, height=6cm]{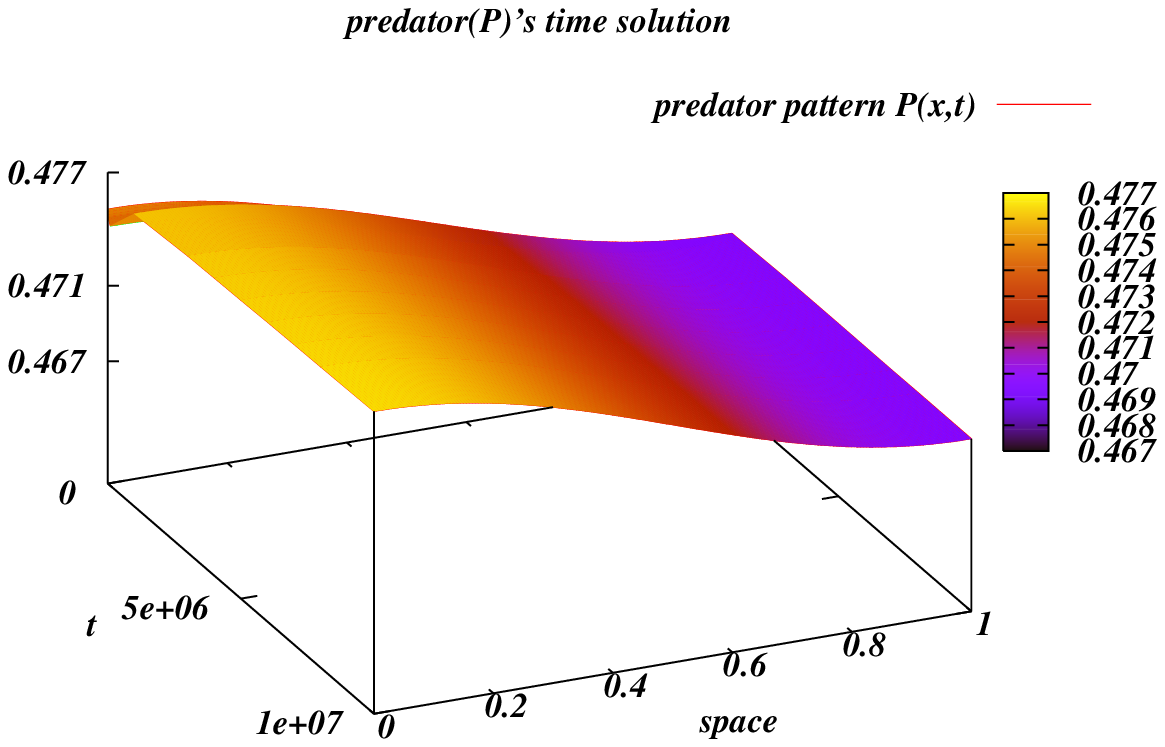}
}\\
\caption{\label{fig:d202} ($d_1$:0.005,$d_2$:0.2,$d_c$:0.271) Upperleft: The prey solution $N(x,t)$ with
      respect to time and space when $d_2 < d_c$. Prey pattern shows the
      convergence to the equilibrium solution $N$ as time increases.
      Upperright: The predator solution $P(x,t)$ with respect to space when
      $d_2 < d_c$. Predator pattern shows the convergence to the
      equilibrium solution $P$ as time
      increases.($d_1$:0.005,$d_2$:0.32,$d_c$:0.271) 
      LowerLeft: The prey solution $N(x,t)$ with
      respect to time and space when $d_2 > d_c$. Prey pattern shows the
      deviation from the equilibrium solution $N$ as time increases.
      LowerRight: The predator solution $P(x,t)$ with respect to space when
      $d_2 > d_c$. Predator pattern shows the deviation from the
      equilibrium solution $P$ as time increases.}
\end{figure}
\newpage
\begin{figure}[hbt]
  \includegraphics[width=10cm]{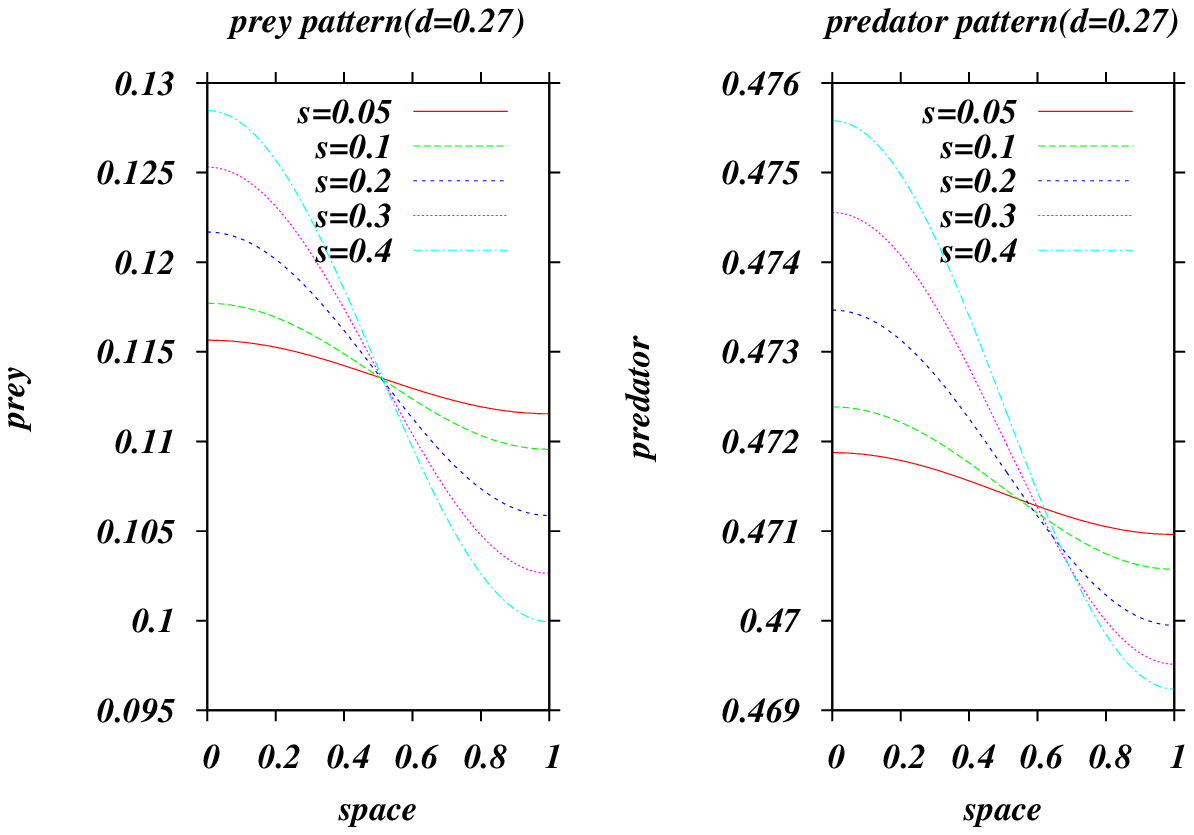}
\vskip 0.3cm
  \includegraphics[width=10cm]{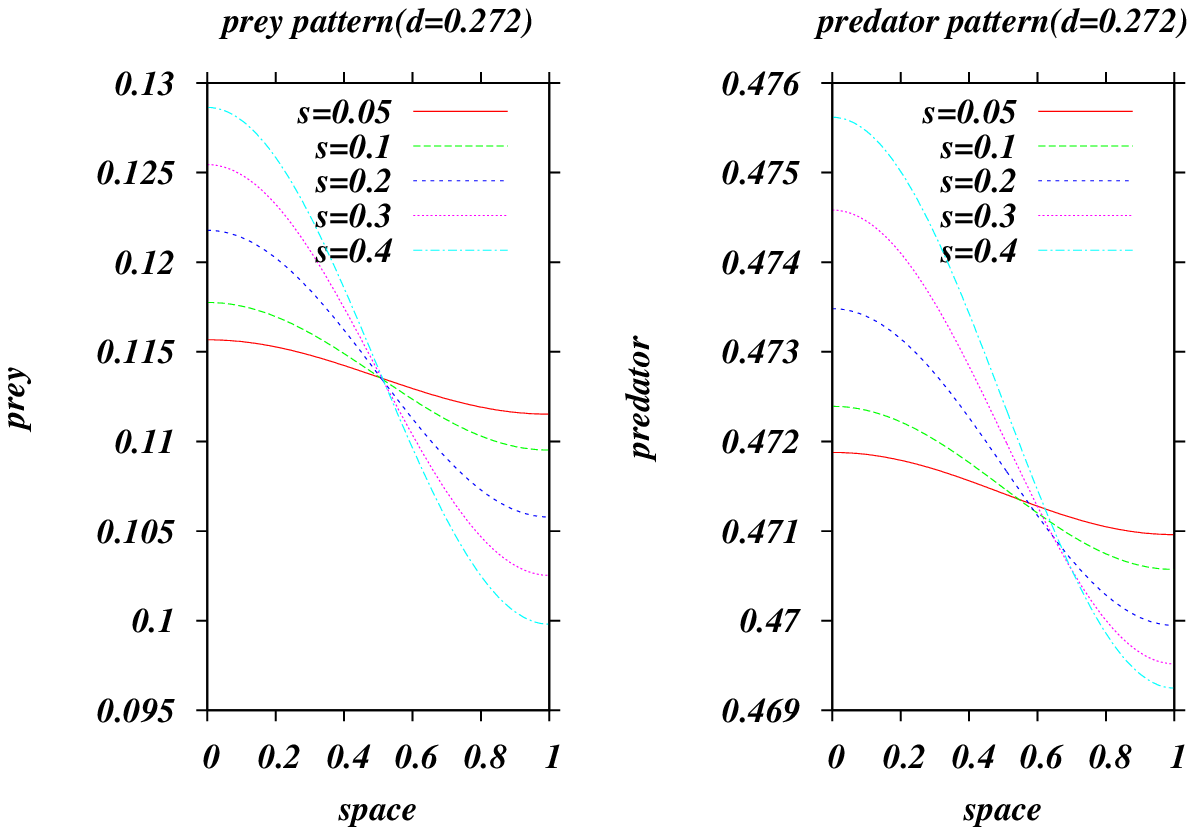}
  \caption{\label{fig:varying-s} Upper: The prey/predator solution
    pattern $N(x,t), P(x,t)$ when $d_2 < d_c$ with varing $s$. 
    ($d_1$:0.005,$d_2$:0.27,$d_c$:0.271)
    Lower: The prey/predator solution pattern $N(x,t),P(x,t)$ when $d_2 > d_c$. 
    ($d_1$:0.005,$d_2$:0.272,$d_c$:0.271)}
\end{figure}
\newpage
\begin{figure}[hbt]
  \includegraphics[height=6cm]{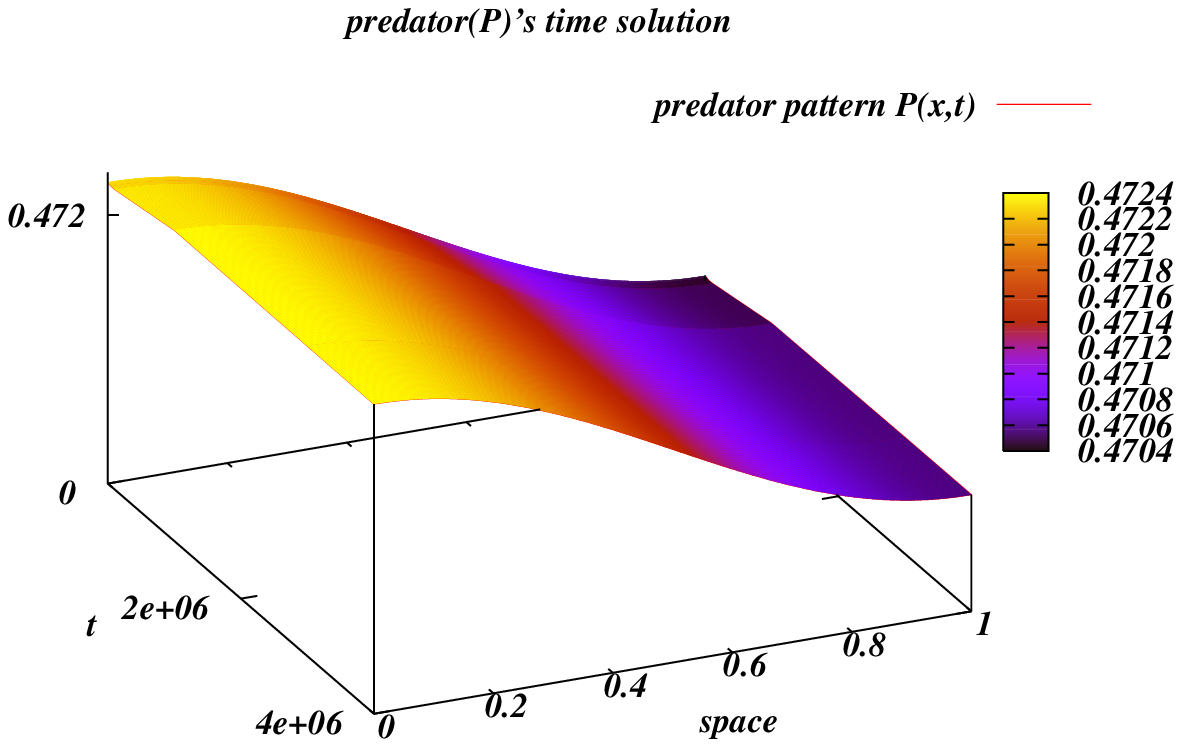}
\hskip 0.3cm
  \includegraphics[height=6cm]{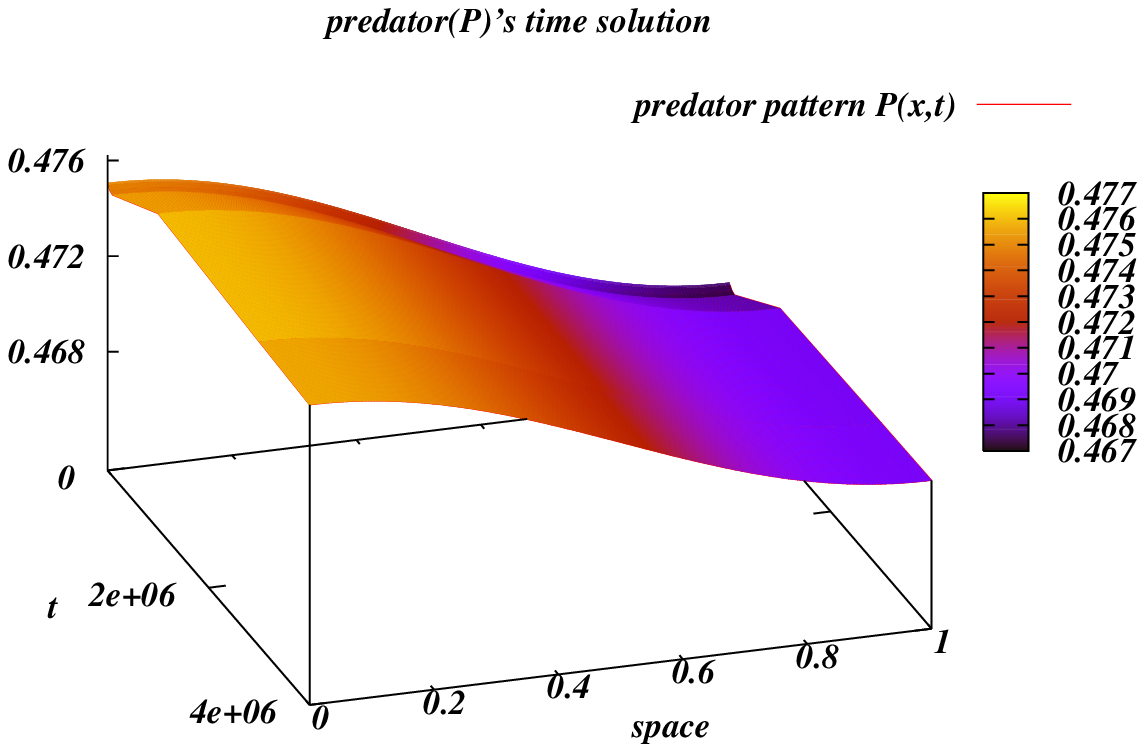}
\vskip 0.3cm
  \includegraphics[height=6cm]{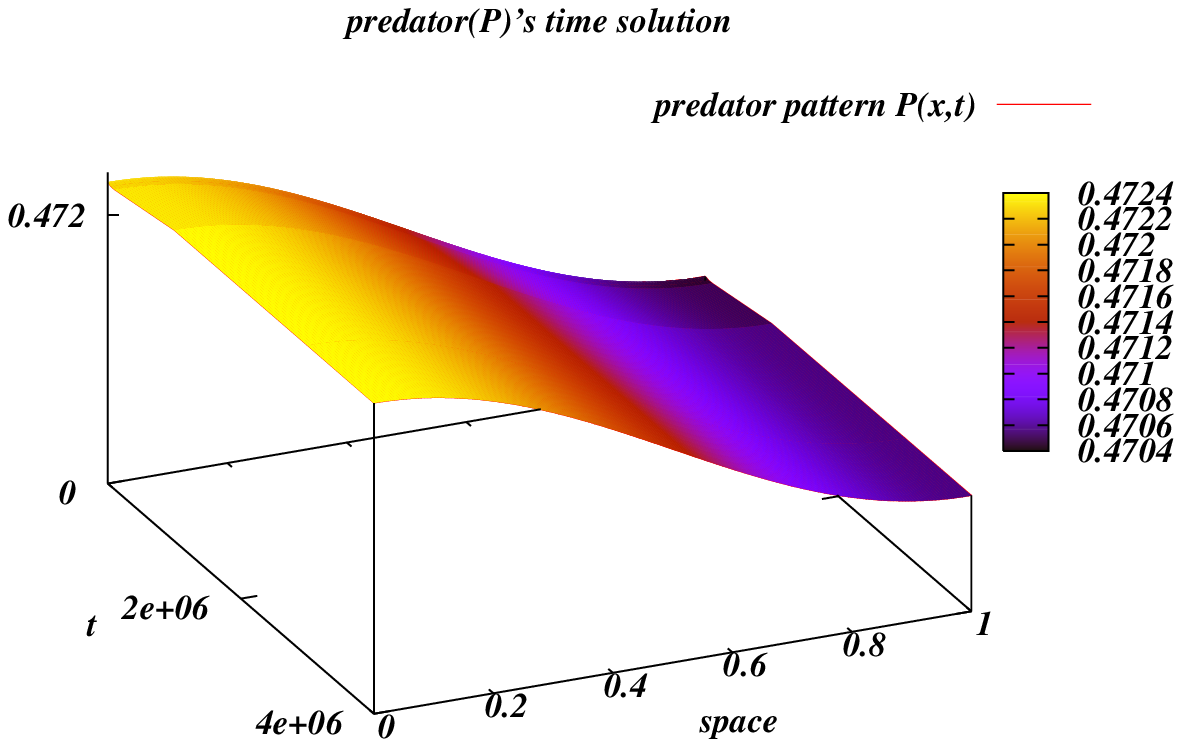}
\hskip 0.3cm
  \includegraphics[height=6cm]{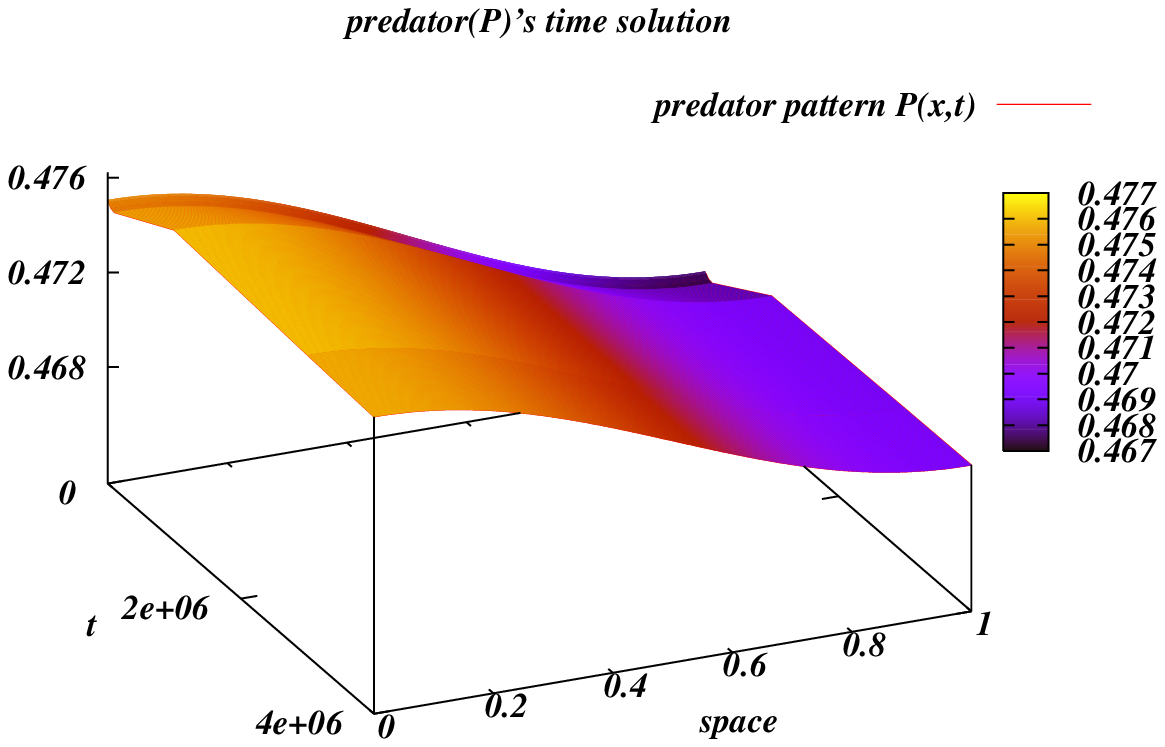}
  \caption{\label{fig:merge-as-preda} Upper Left: The predator solution
    pattern $P(x,t)$ when $s=0.1, d_2 < d_c$. Upper Right: The predator
    solution pattern $P(x,t)$ when $s=0.4,d_2 <d_c$. ($d_1$:0.005,$d_2$:0.27,$d_c$:0.271)
    Lower Left: The predator solution
    pattern $P(x,t)$ when $s=0.1, d_2 > d_c$. Upper Right: The predator
    solution pattern $P(x,t)$ when $s=0.4,d_2 > d_c$. ($d_1$:0.005,$d_2$:0.272,$d_c$:0.271)}
\end{figure}
\newpage
\begin{figure}[hbt]
    \includegraphics[height=6cm]{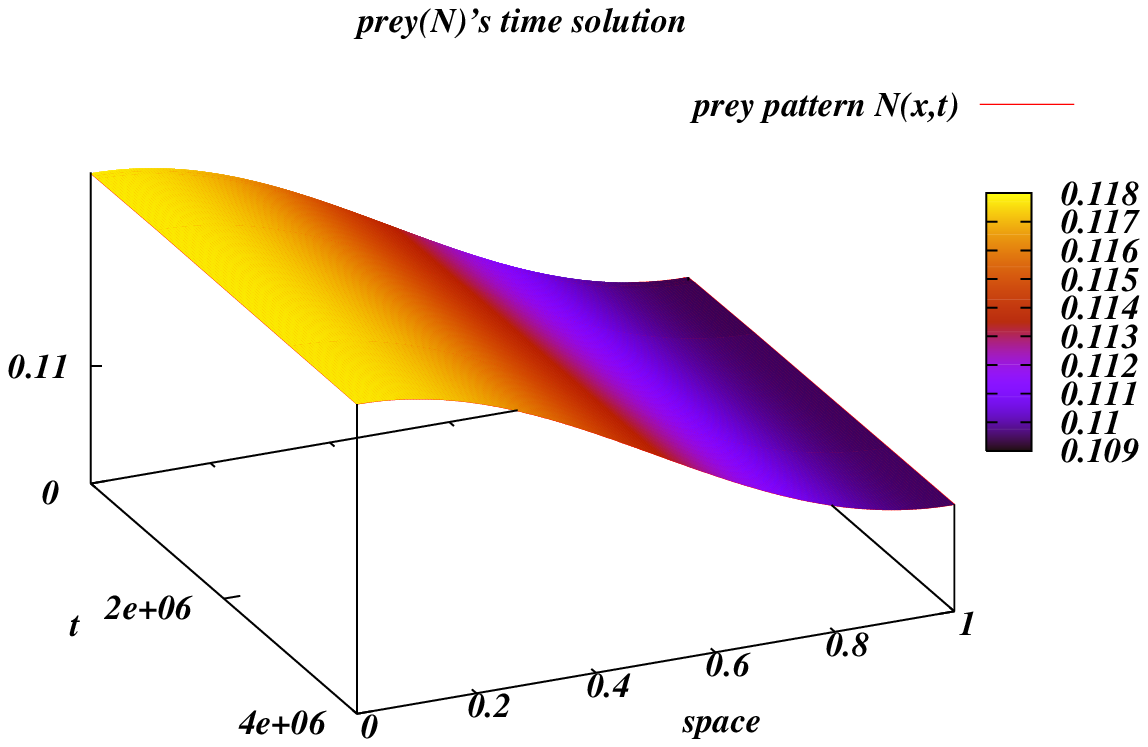}
\hskip 0.3cm
    \includegraphics[height=6cm]{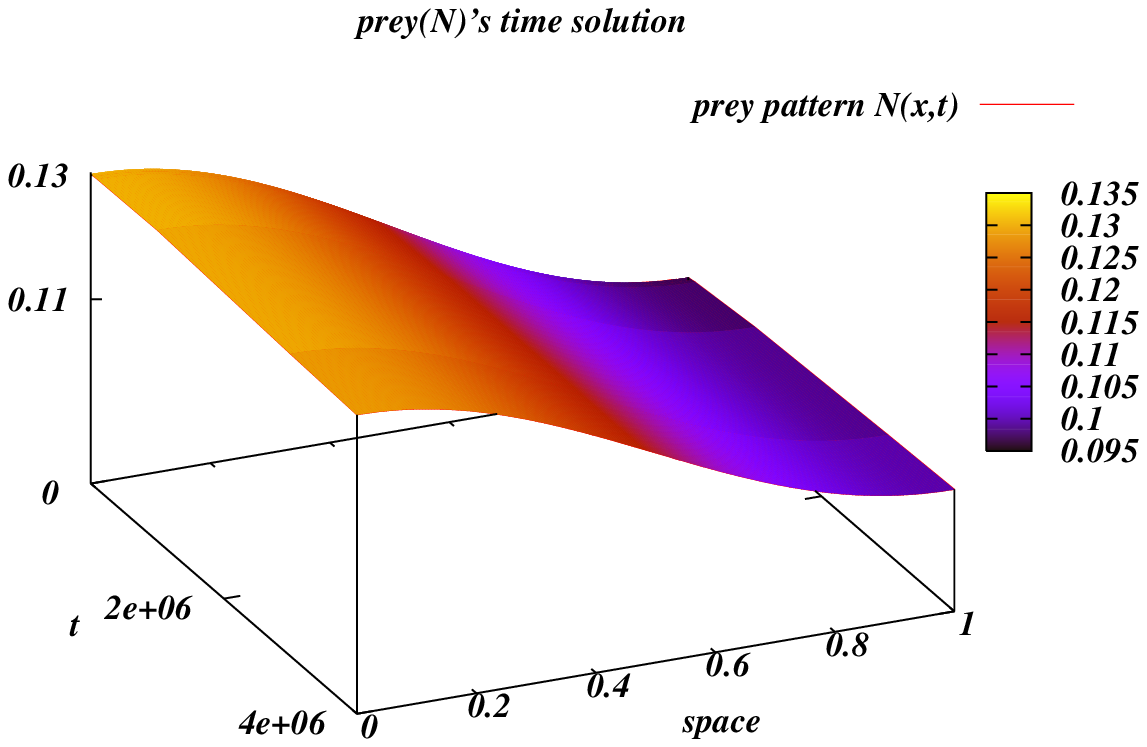}
\vskip 0.3cm
    \includegraphics[height=6cm]{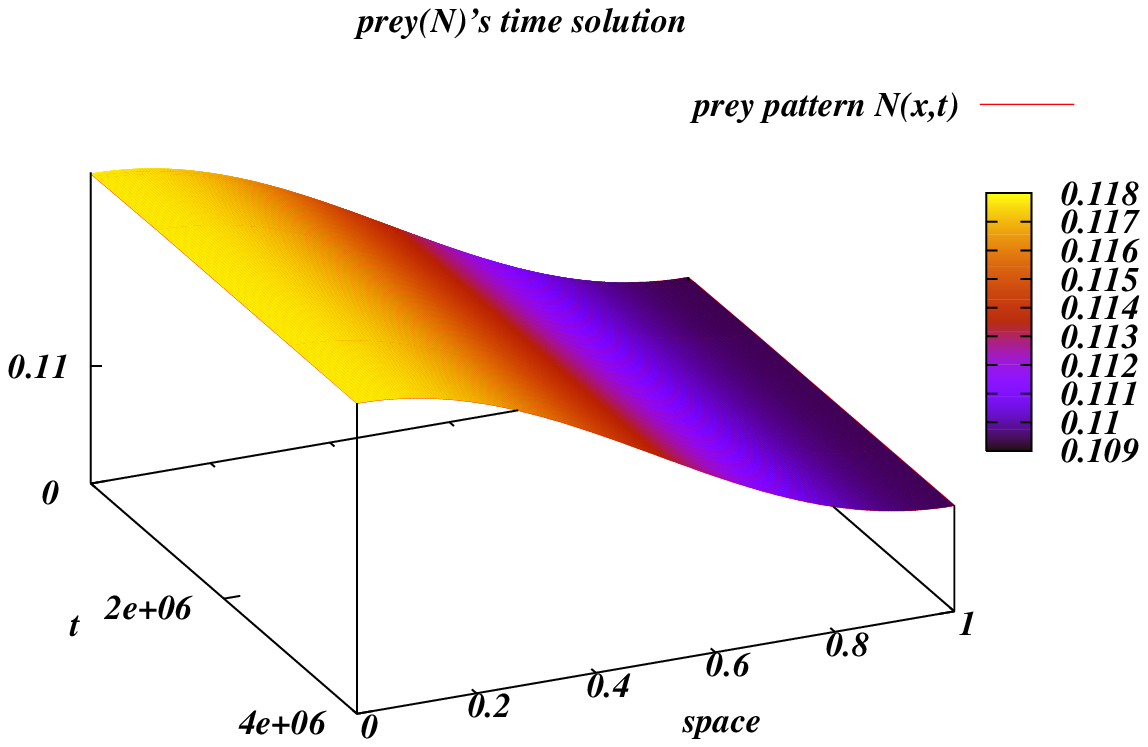}
\hskip 0.3cm
    \includegraphics[height=6cm]{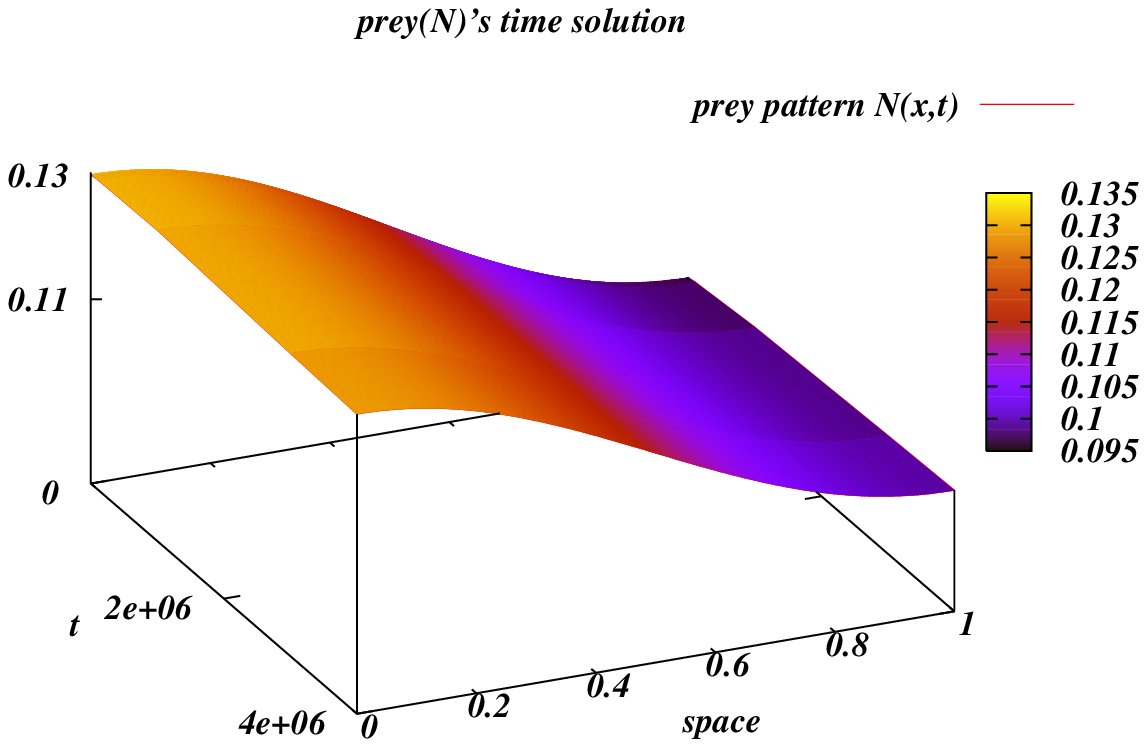}
  \caption{\label{fig:merge-as-prey} Upper Left: The prey solution
    pattern $N(x,t)$ when $s=0.1, d_2 < d_c$. Upper Right: The prey
    solution pattern $N(x,t)$ when $s=0.4,d_2 <d_c$. ($d_1$:0.005,$d_2$:0.27,$d_c$:0.271)
    Lower Left: The prey solution
    pattern $N(x,t)$ when $s=0.1, d_2 > d_c$. Upper Right: The prey
    solution pattern $N(x,t)$ when $s=0.4,d_2 > d_c$. ($d_1$:0.005,$d_2$:0.272,$d_c$:0.271)}
\end{figure}

\newpage
\begin{itemize}
\item Figure 1 $d_1$ and $d_2$ plot, from equation (\ref{3.5})

\item Figure 2 Left: The prey solution at $x$=0.25
with respect to time, the constant line represents $N_e=\overline{N}$ and the two
solid lines represent two different $d_2$ values. Right: The
predator solution at $x$=0.25 with respect to time, the constant
line represents $P_e=\overline{P}$ and the two solid lines represent two different
$d_2$ values.

\item Figure 3($d_1$:0.005,$d_2$:0.2,$d_c$:0.271) Upperleft: The prey solution $N(x,t)$ with
      respect to time and space when $d_2 < d_c$. Prey pattern shows the
      convergence to the equilibrium solution $N$ as time increases.
      Upperright: The predator solution $P(x,t)$ with respect to space when
      $d_2 < d_c$. Predator pattern shows the convergence to the
      equilibrium solution $P$ as time
      increases.($d_1$:0.005,$d_2$:0.32,$d_c$:0.271) 
      LowerLeft: The prey solution $N(x,t)$ with
      respect to time and space when $d_2 > d_c$. Prey pattern shows the
      deviation from the equilibrium solution $N$ as time increases.
      LowerRight: The predator solution $P(x,t)$ with respect to space when
      $d_2 > d_c$. Predator pattern shows the deviation from the
      equilibrium solution $P$ as time increases.

\item Figure 4 Left: The prey/predator solution
    pattern $N(x,t), P(x,t)$ when $d_2 < d_c$ with varing $s$. 
    ($d_1$:0.005,$d_2$:0.27,$d_c$:0.271)
    Right: The prey/predator solution pattern $N(x,t),P(x,t)$ when $d_2 > d_c$. 
    ($d_1$:0.005,$d_2$:0.272,$d_c$:0.271)

\item Figure 5 Upper Left: The predator solution
    pattern $P(x,t)$ when $s=0.1, d_2 < d_c$. Upper Right: The predator
    solution pattern $P(x,t)$ when $s=0.4,d_2 <d_c$. ($d_1$:0.005,$d_2$:0.27,$d_c$:0.271)
    Lower Left: The predator solution
    pattern $P(x,t)$ when $s=0.1, d_2 > d_c$. Upper Right: The predator
    solution pattern $P(x,t)$ when $s=0.4,d_2 >
    d_c$. ($d_1$:0.005,$d_2$:0.272,$d_c$:0.271)

\item  Figure 6 Upper Left: The prey solution
    pattern $N(x,t)$ when $s=0.1, d_2 < d_c$. Upper Right: The prey
    solution pattern $N(x,t)$ when $s=0.4,d_2 <d_c$. ($d_1$:0.005,$d_2$:0.27,$d_c$:0.271)
    Lower Left: The prey solution
    pattern $N(x,t)$ when $s=0.1, d_2 > d_c$. Upper Right: The prey
    solution pattern $N(x,t)$ when $s=0.4,d_2 > d_c$. ($d_1$:0.005,$d_2$:0.272,$d_c$:0.271)
\end{itemize}

\end{document}